\newif\ifspringer
\newif\ifelsevier
\elseviertrue 

\ifspringer
\RequirePackage{fix-cm} 
\documentclass[smallcondensed,nospthms]{svjour3}
\usepackage[paperwidth=15.51cm,paperheight=23.51cm,left=1.75cm,right=1.75cm,top=2cm,bottom=2cm]{geometry} 
\fi

\ifelsevier
\documentclass[preprint,3p,times]{elsarticle}
\makeatletter
\def\ps@pprintTitle{%
   \let\@oddhead\@empty
   \let\@evenhead\@empty
   \def\@oddfoot{\reset@font\hfil\thepage\hfil}
   \let\@evenfoot\@oddfoot
}
\makeatother
\fi

\usepackage[english]{babel}
\usepackage[utf8]{inputenc}
\usepackage{pdflscape} 
\usepackage{amssymb}
\usepackage{amsmath}
\usepackage{amsthm}
\usepackage{mathtools}
\ifspringer
\usepackage{bm} 
\fi
\usepackage[normalem]{ulem}
\usepackage{enumitem}
\usepackage{marvosym}
\usepackage[dvipsnames,svgnames]{xcolor}
\usepackage{array}
\usepackage{multirow}
\ifspringer
\usepackage{mathptmx} 
\fi
\ifelsevier
\usepackage[nodots]{numcompress} 
\fi

\usepackage{graphicx}
\ifspringer
\usepackage[FIGBOTCAP,TABTOPCAP,scriptsize]{subfigure} 
\fi
\ifelsevier
\usepackage[FIGBOTCAP,TABTOPCAP]{subfigure} 
\fi
\usepackage{tikz}
\usetikzlibrary{calc,shapes,arrows,backgrounds}
\tikzstyle{every picture}+=[remember picture]
\usepackage{pgfplots} 

\usepackage{epstopdf} 
\usepackage{hyperref}
\hypersetup{colorlinks=true,linkcolor=blue,citecolor=blue,filecolor=blue,urlcolor=blue}
\usepackage[all]{hypcap}

\mathtoolsset{showonlyrefs} 
\allowdisplaybreaks 
\ifspringer
\smartqed 
\fi
\ifelsevier
\biboptions{sort&compress,numbers}
\fi

\usepackage{xcolor,cancel}

\usepackage{algorithm}
\usepackage{algpseudocode}

\usepackage{xspace}
\makeatletter
\DeclareRobustCommand\onedot{\futurelet\@let@token\@onedot}
\newcommand{\@onedot}{\ifx\@let@token.\else.\null\fi\xspace}
\newcommand{\eg}{{e.g}\onedot} 
\newcommand{\ie}{{i.e}\onedot}

\makeatother

\newtheorem{prop}{Proposition}
\newtheorem{rem}{Remark}
\newtheorem{cor}{Corollary}

\theoremstyle{definition}
\newtheorem{defn}{Definition}
\newtheorem{exmp}{Example}

\newcommand{\vectbi}[1]{\boldsymbol{#1}} 
\newcommand{\vect}[1]{\vectbi{#1}}

\DeclareMathOperator{\supp}{supp}

\renewcommand{\leq}{\leqslant}
\renewcommand{\geq}{\geqslant}

\newcommand{\KKK}{\mathcal{K}}
\newcommand{\PPP}{\mathcal{P}}

\ifspringer
\DeclareMathAlphabet{\mathcalb}{OMS}{cmsy}{b}{n} 
\DeclareMathAlphabet{\mathcal}{OMS}{cmsy}{m}{n} 

\fi
\ifelsevier

\fi

\newcommand{\bM}{\vect{M}}

\newcommand{\bDelta}{\vect{\Delta}}

\newlength{\casesvsep}
\newcommand{\figpath}{figure}
\graphicspath{{\figpath/}}

\makeatletter
\newcommand*{\shifttext}[2]{%
\settowidth{\@tempdima}{#2}%
\makebox[\@tempdima]{\hspace*{#1}#2}%
}
\makeatother

\newcommand{\qtext}[1]{``#1''}

\ifspringer

\fi
\ifelsevier

\fi

\ifspringer
\journalname{\dots}
\fi

\makeatletter
\providecommand{\doi}[1]{%
  \begingroup
    \let\bibinfo\@secondoftwo
    \urlstyle{rm}%
    \href{http://dx.doi.org/#1}{%
      doi:\discretionary{}{}{}%
      \nolinkurl{#1}%
    }%
  \endgroup
}
\makeatother

\begin{document}

\newcommand{\titletext}{
On multi-degree splines
}
\newcommand{\titlerunningtext}{
On multi-degree splines
}

\newcommand{\abstracttext}{
Multi-degree splines are piecewise polynomial functions having sections of different degrees.
For these splines, we discuss the construction of a B-spline basis by means of integral recurrence relations, extending the class of multi-degree splines that can be derived by existing approaches.
We then propose a new alternative method for constructing and evaluating the B-spline basis, based on the use of so-called transition functions.
Using the transition functions we develop general algorithms for knot-insertion, degree elevation and
conversion to Bézier form, essential tools for applications in geometric modeling.
We present numerical examples and briefly discuss how the same idea can be used in order to construct geometrically continuous multi-degree splines.
}

\ifelsevier
\newcommand{\separ}{\sep}
\fi

\newcommand{\keytext}{
Multi-degree spline \separ B-spline basis \separ Transition function \separ Geometric modeling \separ Knot-insertion \separ Degree elevation}

\newcommand{\MSCtext}{
65D07 \separ 65D17 \separ 41A15 \separ 68W25
}


\ifelsevier
\begin{frontmatter}

\title{\titletext}

\author[label2]{Carolina Vittoria Beccari}
\ead{carolina.beccari2@unibo.it}
\author[label2]{Giulio Casciola}
\ead{giulio.casciola@unibo.it}
\author[label2]{Serena Morigi}
\ead{serena.morigi@unibo.it}

\address[label2]{Department of Mathematics, University of Bologna,
Piazza di Porta San Donato 5, 40126 Bologna, Italy}

\begin{abstract}
\abstracttext
\end{abstract}

\begin{keyword}
\keytext
\MSC[2010]\MSCtext
\end{keyword}

\end{frontmatter}
\fi


\section{Introduction}
\label{sec:intro}

Multi degree splines (MD-splines, for short) are piecewise functions comprised of polynomial segments of different degrees.
They were proposed in the seminal paper \cite{SZS2003} and they have been a subject of study in several recent works \cite{WD2007,SW2010a,SW2010b,Li2012,SW2013,Shenetal2016}.
A more general setting was considered in \cite{BMuhl2003}, where the section spaces belong to the class of Extended Chebyshev spaces and are not constrained to be of the same dimension.

In addition to the knot intervals and control points, to model a shape MD-splines leverage on one more parameter, namely the degree.
The degree can be chosen locally to get the best shape fitting,
thus allowing to use less control points than those necessary with \emph{conventional} splines (the latter being intended as spline spaces where every piece is spanned by polynomials of the same degree).
This is illustrated by several examples in \cite{Shenetal2016,SZS2003}, where the same curve is designed using the Bézier representation, conventional splines and MD-splines.
At the same time, MD-splines reduce to conventional splines when all segments are of the same degree, thus generalizing the traditional approach.
Recently, the concept of working locally with polynomials of different degrees has also been introduced with a view to application in Isogeometric Analysis, both in the context of T-splines \cite{Liu2015} and polar splines \cite{TOSH2017}.

Previous proposals of MD-splines differ in the way the connections between polynomial segments are handled. Most of them do not include the possibility of using multiple knots, and as a consequence, they do not allow to control the degree of continuity as we do with conventional splines.
In particular the constructions proposed in \cite{Li2012} and \cite{SW2010b} yield splines that are exactly $C^{d-1}$-continuous at the join between two segments of same degree $d$. Between two segments of degrees $d_{i-1}$ and $d_{i}$, instead, continuity of order $C^{\min(d_{i-1},d_{i})}$ is attained by the method in \cite{SW2010b}, while the splines in \cite{Li2012} can be $C^1$-continuous only. However, both constructions do not allow for using multiple knots in order to reduce the continuity.
Only in  \cite{SW2010a} knots of multiplicity greater than one are admitted. As a result,
any order of continuity can be attained, but only up to $C^{\min(d_{i-1},d_{i})-1}$, for both $d_{i-1}\neq d_i$ and $d_{i-1}=d_i=d$.

The construction devised in this paper includes and extends all previously proposed instances of MD-splines,
allowing any continuity from $C^0$ to $C^{\min(d_{i-1},d_{i})}$ at the join.
The degree of continuity is handled by means of multiple knots, thus following the classical approach with conventional splines.
However, the maximum attainable continuity  $C^{\min(d_{i-1},d_{i})}$ is higher, and can be interpreted as having one knot of multiplicity equal to zero.
This peculiarity is indeed a natural consequence of working with spline spaces having sections of different degree (as well as different nature such as in the case of spaces spanned locally by trigonometric, polynomial or hyperbolic functions) or geometric continuity \cite{DynMicchelli1989}.

In \cite{SW2010b} and \cite{SW2010a} the B-spline basis is generated by means of integral recurrence relations.
This type of formul\ae \; provides an elegant construction for the B-spline basis and a convenient way to derive its properties, nevertheless the integral definition of the basis functions is not easily computable.
This is well understood also in the context of conventional splines, where
Cox-de Boor recurrence relation is the preferred tool for evaluating the B-spline basis.

In this paper, we first construct MD-splines by means of an integral recurrence relation which generalizes the approaches in \cite{SW2010a,SW2010b}.
We also provide a Cox-de Boor type evaluation algorithm for a specific class of MD-splines which are limited to $C^1$ smoothness.
However, to the best of our knowledge, a similar recurrence formula does not exist in the general MD-spline framework.
This will prompt us to introduce a new approach for computing the basis functions.
We will show that the elements of the B-spline basis can be expressed in terms of another basis, which is composed of so-called \emph{transition functions}, the latter being very easy to compute.
In fact, the transition functions are simply calculated as the solution of an Hermite interpolation problem, which always admits a unique solution in a MD-spline space \cite{BMuhl2003}.
Furthermore we will show how commonly performed operations, such as knot insertion, degree elevation and conversion to Bézier form,
can simply be accomplished relying on the transition functions.
In previous papers, transition functions were introduced in the context of conventional splines \cite{ABC2013a,BCR2013a} and piecewise Chebyshevian splines with sections all of the same dimension \cite{BCR2017}.
The hurdle posed by MD-splines is mostly in the need for handling spaces of different dimension, which entails that the support of the basis functions is nontrivially defined.
To overcome this difficulty, we exploit two auxiliary knot partitions, enabling us to clearly identify the points where the basis functions start and terminate and their continuity at these locations.
This idea is at the basis of both the proposed generalized integral recurrence relation and the derivation of the transition functions.

In closing we briefly discuss how to extend the proposed construction to the wider framework of geometrically continuous MD-splines.
In this setting the continuity conditions between adjacent spline pieces are expressed by means of proper connection matrices. The entries of these matrices provide additional degrees of freedom that can be exploited as shape parameters in CAGD applications \cite{DynMicchelli1989,BCM2016}.\\

The remainder of the paper is organized as follows. In Section \ref{sec:definition} we define the considered MD-spline spaces and introduce our setting and notation. The two auxiliary knot partitions, which will be essential throughout the paper, and the construction of the B-spline basis by means of an integral recurrence relation are discussed in Section \ref{sec:integral_rec}. In Section \ref{sec:computation} we address the problem of computing such a basis in an alternative way via transition functions.
In Section \ref{sec:modeling_tools},
the usual modeling tools, including knot insertion, degree elevation and
conversion to Bézier form, are then derived in terms of transition functions.
Finally Section \ref{sec:geometric} is devoted to discussing how the proposed approach can be extended to generate MD-splines that are geometrically continuous and Section \ref{sec:examples} presents simple examples of applications to geometric modeling.

\section{Multi-degree spline spaces}
\label{sec:definition}
A MD-spline is a function defined on an interval $[a,b]$ and composed of pieces of polynomial functions of different
degrees defined on subintervals and joined at their endpoints with a suitable degree of smoothness.

To construct MD-splines, we introduce the following setting.

\noindent Let $[a,b]$ be a bounded and closed interval, and $\bDelta\coloneqq\left\{x_i\right\}_{i=1}^{q}$ be a set of \emph{break-points}
such that $a \equiv x_{0} < x_{1} < \ldots <x_{q} < x_{q+1} \equiv b$.
The polynomial pieces are defined on the subintervals $\left[x_i,x_{i+1}\right]$, $i=0,\dots,q$.
Let $\mathbf{d}\coloneqq\left(d_0,\ldots,d_q\right)$ be a vector of positive integers, where $d_i$ is the degree of
the polynomial defined on the interval $\left[x_i,x_{i+1}\right]$.
Then two adjacent polynomials defined respectively on the subintervals $\left[x_{i-1},x_i\right]$ and $\left[x_i,x_{i+1}\right]$
join at the break-point $x_i$ with continuity $C^{k_i}$ where $k_i$ is a nonnegative integer such that
$$
0 \leq k_i \leq \left\{
\begin{array}{lll}
 min(d_{i-1},d_i), & if & d_{i-1} \neq d_i, \\
 d_i-1, & if & d_{i-1} = d_i. \\
\end{array} \right.
$$
The vector $\KKK\coloneqq\left(k_1,\dots,k_q\right)$ of nonnegative integers determines
the degree of smoothness.

\noindent The set $S(\PPP_{\mathbf{d}},\KKK,\bDelta)$ of multi degree splines is defined as follows.

\begin{defn}[Multi-degree splines]\label{def:QCS}
Given a partition $\bDelta = \left\{x_i\right\}_{i=1}^q$ on the bounded and closed interval $[a,b]$,
the associated sequence of polynomial degrees $\mathbf{d}$, and the corresponding sequence $\KKK$ of degrees of smoothness,
the set of \emph{multi-degree splines} is given by
\begin{align}
S(\PPP_{\mathbf{d}},\KKK,\bDelta) \coloneqq & \left\{ f \,\big|\, \mbox{there exist } p_i\in\PPP_{d_i}, i=0,\dots,q,  \mbox{ such that:} \right. \\
&
\begin{minipage}[b]{0.75\linewidth}
\begin{enumerate}[label=\roman*)]
 \item $f(x)=p_i(x)$ for $x\in [x_i,x_{i+1}], \, i=0,\dots,q$;
 \item $D^\ell p_{i-1}(x_i)=D^\ell p_{i}(x_i)$ for $\ell=0,\dots,k_i$, \, $i=1,\dots,q \left. \vphantom{\big|} \right\}.$
\end{enumerate}
\end{minipage}
\end{align}
\end{defn}

It follows from standard arguments that the set $S(\PPP_{\mathbf{d}},\KKK,\bDelta)$ of multi-degree splines in Definition
\ref{def:QCS} is a function space of dimension $d_0+1+K_s$ with $K_s\coloneqq \sum_{i=1}^q (d_i-k_i)$ or, equivalently,
$d_q+1+K_t$ with $K_t\coloneqq \sum_{i=1}^q (d_{i-1}-k_i)$.
From now on, for simplicity, we will denote the dimension of the spline space with $K$.
In case $d_i=d$, for each $i=0,\dots,q$, and a fixed positive integer $d$, then $S(\PPP_{\mathbf{d}},\KKK,\bDelta)$ reduces
to the conventional spline space.
A useful property of these spaces is that the zero count for conventional splines carries over to MD-splines. In particular, the number of zeroes of a spline $f$ in $[x_p,x_r]$, counting multiplicities, is bounded as follows \cite{BMuhl2003,Buchwald_thesis}:
\begin{equation}\label{eq:zeros}
Z\left(f,[x_p,x_r]\right)\leq \sum_{i=p}^{r-1} (d_i+1)-\sum_{i=p+1}^{r-1} (k_i+1) - 1=d_p+\sum_{i=p+1}^{r-1}(d_i-k_i)-1,
\end{equation}
namely it is smaller than the dimension of the spline space restricted to the considered interval.

\section{Multi-degree B-spline bases through integral recurrence relation}
\label{sec:integral_rec}

For conventional degree-$d$ splines, associated with a given extended knot vector,
each basis function has compact support defined by $d+1$ consecutive knot intervals.
For the construction of the B-spline basis for the multi-degree spline space $S(\PPP_{\mathbf{d}},\KKK,\bDelta)$,
it is convenient to consider two different extended knot vectors $\bDelta^*_s=\{s_j\}$ and  $\bDelta^*_t=\{t_j\}$
such that the $i$-th B-spline basis function $N_{i,m}$, with $m\coloneqq \max_i\{d_i\}$,
\qtext{starts} at $s_i \in \bDelta^*_s$ and \qtext{terminates} in $t_i \in \bDelta^*_t$.
Hence its support, $\supp N_{i,m}$, is the interval $[s_i,t_i]$ defined by a sequence of consecutive break-point intervals.

\begin{defn}[Extended partitions]\label{def:part_estesa}
The set of knots $\bDelta_s^*\coloneqq\left\{s_j\right\}_{j=1}^{K}$, with $K\coloneqq d_0+K_s+1$ and $K_s\coloneqq \sum_{i=1}^q (d_i-k_i)$, is called the \emph{left extended partition} associated with $S(\PPP_{\mathbf{d}},\KKK,\bDelta)$ if and only if:
\begin{enumerate}[label=\roman*)]
 \item $s_1 \leq s_2 \leq \dots \leq s_{K}$;
 \item $s_{d_0+1} \equiv a$;
 \item $\left\{s_{d_0+2}, \dots, s_{K}\right\} \equiv
        \{ \underbrace{x_1,\dots,x_1}_{d_1-k_1 \text{ times}}, \dots, \underbrace{x_q,\dots,x_q}_{d_q-k_q \text{ times}} \}$.
\end{enumerate}
Similarly, the set of knots $\bDelta_t^*\coloneqq\left\{t_j\right\}_{j=1}^{K}$, with $K\coloneqq d_q+K_t+1$ and $K_t\coloneqq \sum_{i=1}^q (d_{i-1}-k_i)$, is called
the \emph{right extended partition} associated with $S(\PPP_{\mathbf{d}},\KKK,\bDelta)$ if and only if:
\begin{enumerate}[label=\roman*)]
 \item $t_1 \leq t_2 \leq \dots \leq t_{K}$;
 \item $t_{K-d_q} \equiv b$;
 \item $\left\{t_{1}, \dots, t_{K-d_q-1}\right\} \equiv
        \{ \underbrace{x_1,\dots,x_1}_{d_0-k_1 \text{ times}}, \dots, \underbrace{x_q,\dots,x_q}_{d_{q-1}-k_q \text{ times}} \}$.
\end{enumerate}
\end{defn}

For simplicity and without loss of generality, we will limit our discussion to the case of \emph{clamped} partitions, \ie extended partitions with the two extreme break-points repeated $d_0 + 1$ times in $\bDelta_s^*$, and $d_q+1$ times in $\bDelta_t^*$,
that is $s_1= \ldots =s_{d_0+1}=x_0$ and $t_{K-d_q}= \ldots =t_{K}=x_{q+1}$.

The set of multi-degree B-spline functions $\{N_{i,m}(x)\}_{i=1}^{K}$ can be generated by the following integral recurrence relation.

\begin{defn}[Basis functions]\label{def:basisf}
Let $m\coloneqq \max_i\{d_i\}$.
The function sequence $\{N_{i,n}(x)\}$
is defined over $\bDelta^*_s$ and $\bDelta^*_t$ for any recurrence step $n=0,\ldots,m$ and $i=m+1-n,\ldots,K$.

Each $N_{i,n}$ has support $\supp N_{i,n}=[s_i,t_{i-m+n}]$,
and is defined on each break-point interval $[x_j,x_{j+1}) \subset [s_i,t_{i-m+n}]$ with $s_i<t_{i-m+n}$
as follows:
\begin{equation}\label{def:int_rec}
\displaystyle
N_{i,n}(x)\coloneqq \left \{
\begin{array}{ll}
1, \qquad x_j \leq x < x_{j+1} & n=m-d_j, \vspace{0.2cm}\\
\int_{-\infty}^x\left[\delta_{i,n-1}N_{i,n-1}(u)-\delta_{i+1,n-1}N_{i+1,n-1}(u)\right]du,  & n>m-d_j, \vspace{0.2cm}\\
0,  & otherwise,
\end{array}
\right.
\end{equation}
where
\[\delta_{i,n}\coloneqq\left( \int_{-\infty}^{+\infty} N_{i,n}(x) dx \right)^{-1}.\]

Undefined $N_{i,n}$ functions must be regarded as the zero function.
In addition, like in the de Boor formula for conventional B-spline basis functions ($0/0=0$), we
set $\delta_{i,n} N_{i,n} \coloneqq 0$ when $N_{i,n}(x)=0$. However, in order to satisfy the partition of unity, $\delta_{i,n} N_{i,n}$ should satisfy $\int_{-\infty}^{+\infty}\delta_{i,n}N_{i,n}(x)dx=1$. Therefore, when $N_{i,n}(x)=0$, we set
\begin{equation} \label{cond_N=0}
\int_{-\infty}^x\delta_{i,n}N_{i,n}(u)du\coloneqq \left\{
\begin{array}{ll}
0, & x<s_i,\\
1, & x \geq t_i.\\
\end{array}
\right.
\end{equation}
\end{defn}

From now on we indicate by $ps_{i}$ the index of the break-point associated with the knot $s_i$ and
$pt{_i}$ the index of the break-point associated with the knot $t_i$. With this notation we can state the properties of the above defined B-spline basis.

\begin{defn}[B-spline basis properties]\label{def:B-spline basis}
The B-spline functions $\{N_{i,m}\}_{i=1}^{K}$ of the MD-spline space $S(\PPP_{\mathbf{d}},\KKK,\bDelta)$ built by relation \eqref{def:int_rec}
satisfy the following properties:
\begin{enumerate}[label=\roman*)]
\setlength{\itemsep}{4pt}
\item \label{propty:BS1} \emph{Local Support}: $N_{i,m}(x)=0$ for $x \notin [s_i, t_i]$;
\item \label{propty:BS2} \emph{Positivity}:  $N_{i,m}(x) > 0$ for $x \in (s_i,t_i)$;
\item \label{propty:BS3} \emph{End Point}: $N_{i,m}$ vanishes exactly
\begin{itemize}
	\item $d_{ps_i}-\max \{ j\geq 0 \ | \ s_i=s_{i+j} \} \;$ times at $s_i$,
	\item $d_{pt_i-1}-\max \{ j \geq 0 \ | \ t_{i-j}=t_i \} \;$ times at $t_i$;
\end{itemize}
\item \label{propty:BS4} \emph{Normalization}: $\displaystyle \sum_i N_{i,m}(x) = 1$, $\forall x \in [a,b]$.
\end{enumerate}
\end{defn}

The above property \ref{propty:BS3} immediately follows from the integral recurrence relation \eqref{def:int_rec} and guarantees the linear independence of the constructed functions. This, together with the fact that the number of basis functions generated by \eqref{def:int_rec} equals the dimension of the spline space, yields that the set $\{N_{i,m}\}_{i=1}^K$ is a basis for the space $S(\PPP_{\mathbf{d}},\KKK,\bDelta)$.
Similarly, also properties \ref{propty:BS1}, \ref{propty:BS2} and \ref{propty:BS4} are verified by construction. A detailed proof has already been presented for the splines in \cite{SW2010a} and can be repeated in our case following the same outline.

Any MD-spline $f \in S(\PPP_{\mathbf{d}},\KKK,\bDelta)$ is represented as a linear combination of the B-spline basis functions
$N_{i,m}$, $i=1,\dots,K$, defined in \eqref{def:int_rec}, in the following form
\begin{equation}\label{eq:spline_global}
f(x)=\sum_{i=1}^{K} c_i\,N_{i,m}(x), \quad x\in[a,b],
\end{equation}
and also, locally, as
\begin{equation}\label{eq:spline_local}
f(x)=\sum_{i=\ell-d_j}^{\ell} c_i\,N_{i,m}(x), \quad x\in[x_j,x_{j+1}], \quad s_\ell \leq x_j < \min\left(s_{\ell+1},b\right).
\end{equation}

By construction,  two adjacent segments of a MD-spline with different degrees $d_{i-1}$ and $d_{i}$
join at the break-point $x_{i}$ with continuity
$C^k$, $0\leq k\leq \min\left(d_{i-1},d_{i}\right)$,
and the continuity between two adjacent segments of same degree $d_{i-1}=d_i=d$ is 
$C^k$, $0\leq k\leq d-1$.
This is a potential that goes beyond what is offered by the conventional spline setting and is made plausible through
the concept of break-points with \emph{zero multiplicity}, \ie having no knot $s$ or $t$ associated to them.

The proposed construction is characterized by the use of the two extended partitions $\bDelta_s$ and $\bDelta_t$. This is the main difference with respect to other integral recurrence relations previously proposed for the MD B-spline basis \cite{SW2010b,SW2010a} and allows us to achieve a wider range of continuities at the break-points.
More precisely, compared with the proposal of Changeable Degree splines in \cite{SW2010a}, which are limited up to $C^{\min(d_{i-1},d_{i})-1}$ continuity,
the MD-splines in this paper have higher order of continuity. Moreover, they allow for a control on the degree
of continuity at the break-points, a benefit which is not offered by the other proposals of MD-splines.
In particular, the MD-splines introduced in \cite{SW2010b} limit the continuity between segments
of different degrees $d_{i-1}$ and $d_{i}$ at the highest smoothness $C^{\min(d_{i-1},d_{i})}$, while in \cite{Li2012}
the MD-splines are strictly required to be $C^1$ between two adjacent curve segments with different degrees.

The following example illustrates the notations and the recursive formula for MD-splines.
\begin{figure}[tbh]
\centering
\begin{tikzpicture}[scale=0.8]
\def\d{2}
\def\dd{2}
\def\ks{1.8575}
\def\lxlim{-0.035}
\def\rxlim{9.785}
\def\ylim{7.5}

\def\mks{0.07}

\def\xa{0}
\def\xb{1.39}
\def\xc{4.185}
\def\xd{8.365}
\def\xe{9.785}

\node[anchor = south] at (0,0) {};

\draw[-] (\lxlim,0) -- (\rxlim,0) ;
\draw[yshift=-0.2\baselineskip]	(\xa,0) node[anchor=north] {$x_0$}
		                        (\xb,0) node[anchor=north] {$x_1$}
		                        (\xc,0) node[anchor=north] {$x_2$}
	                            (\xd,0) node[anchor=north] {$x_3$}
	                            (\xe,0) node[anchor=north] {$x_4$};


\filldraw (\xa,0) circle [radius=0.1,fill=black];
\filldraw (\xb,0) circle [radius=0.1,fill=black];
\filldraw (\xc,0) circle [radius=0.1,fill=black];
\filldraw (\xd,0) circle [radius=0.1,fill=black];
\filldraw (\xe,0) circle [radius=0.1,fill=black];

\draw[yshift=-0.7\baselineskip,color=red] (0.5*\xa+0.5*\xb,0) node[anchor=north] {{\small$d_0=1$}}
		                                  (0.5*\xb+0.5*\xc,0) node[anchor=north] {{\small$d_1=2$}}
		                                  (0.5*\xc+0.5*\xd,0) node[anchor=north] {{\small$d_2=4$}}
	                                      (0.5*\xd+0.5*\xe,0) node[anchor=north] {{\small$d_3=2$}};

\draw[yshift=1.75\baselineskip,color=blue] (\xb,0) node[anchor=north] {$C^0$}
		                                   (\xc,0) node[anchor=north] {$C^1$}
		                                   (\xd,0) node[anchor=north] {$C^2$};

\draw[yshift=-0.02\baselineskip](-1.48,0) node[anchor=north] {$\bDelta$};

\draw[yshift=1.73\baselineskip](-1.4,-\d) node[anchor=north, yshift=0.2cm] {$\bDelta^*_s$};

\draw[-] (\lxlim,-\d) -- (\rxlim,-\d);
\draw[yshift=1.5\baselineskip]	(\xa,-\d) node[anchor=north,yshift=0.2cm] {$s_1${\small$\equiv$}$s_2$}
		                        (\xb,-\d) node[anchor=north, yshift=0.2cm] {$s_3${\small$\equiv$}$s_4$}
		                        (\xc,-\d) node[anchor=north, yshift=0.2cm] {$s_5${\small$\equiv$}$s_7$};

\foreach \x in {\xa,\xa+\mks,\xb,\xb+\mks,\xc-\mks,\xc,\xc+\mks}
     		\draw[semithick] (\x,{-\d}) -- (\x,{-\d+0.3});
	
\draw[yshift=-0.47\baselineskip](-1.4,-\dd) node[anchor=north] {$\bDelta^*_t$};

\draw[-] (\lxlim,-\dd) -- (\rxlim,-\dd);
\draw[yshift=-0.6\baselineskip]	(\xb,-\dd) node[anchor=north] {$t_1$}
		                        (\xc,-\dd) node[anchor=north] {$t_2$}
		                        (\xd,-\dd) node[anchor=north] {$t_3${\small$\equiv$}$t_4$}
                                (\xe,-\dd) node[anchor=north] {$t_5${\small$\equiv$}$t_7$};
\foreach \x in {\xb+0.5*\mks,\xc,\xd,\xd+\mks,\xe-2*\mks,\xe-\mks,\xe}
     		\draw[semithick] (\x,{-\dd-0.3}) -- (\x,{-\dd});	
\end{tikzpicture}
\label{fig:ex1_part}
\caption{Setting for the spline space in Example \ref{ex:running}.}
\end{figure}
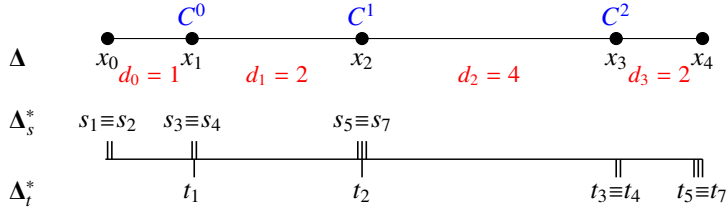
\begin{figure}[tbh]
\center
\begin{tabular}{c}
\begin{tabular}{ccc}
\includegraphics[width=2.0in]{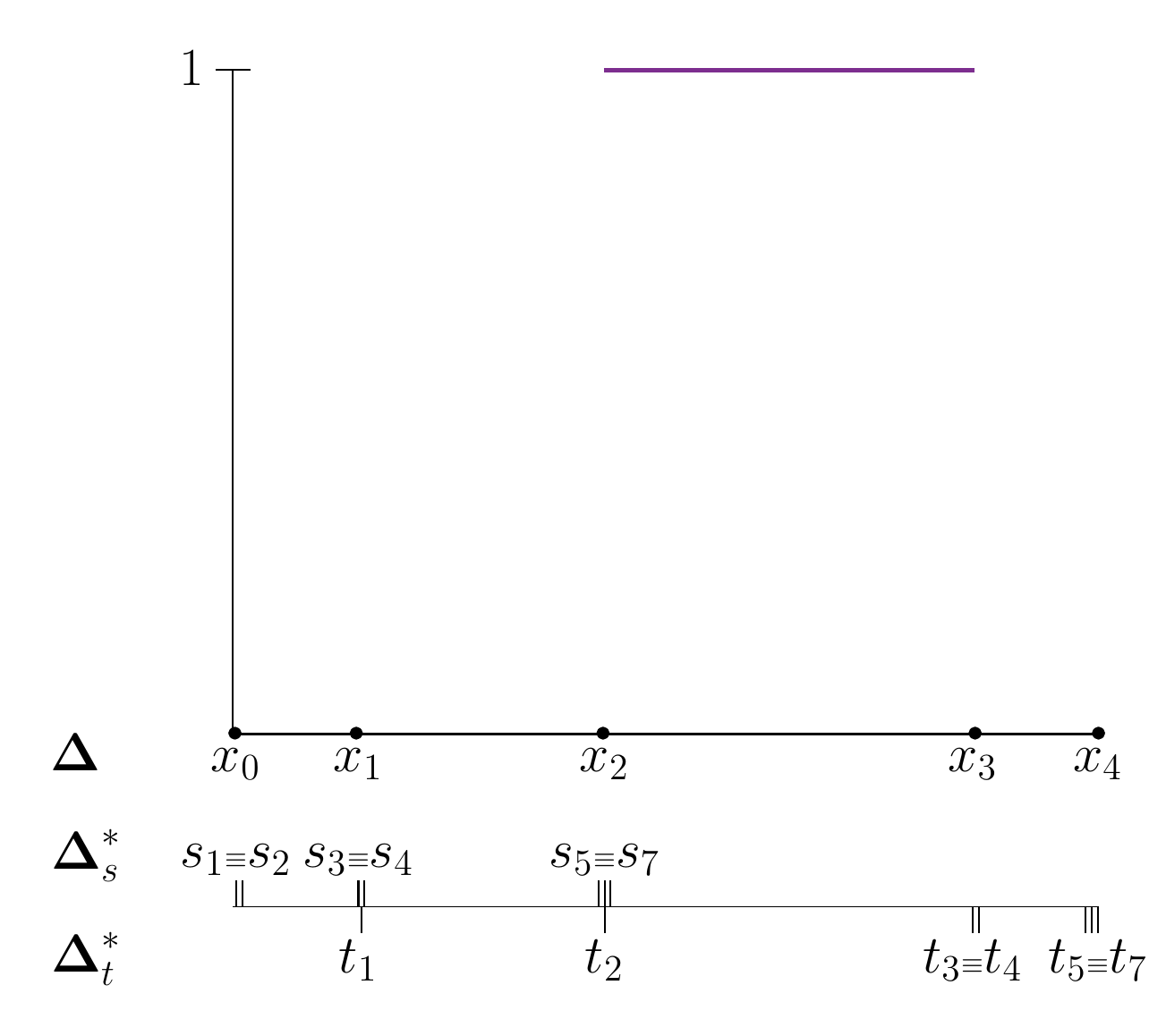} &
\includegraphics[width=2.0in]{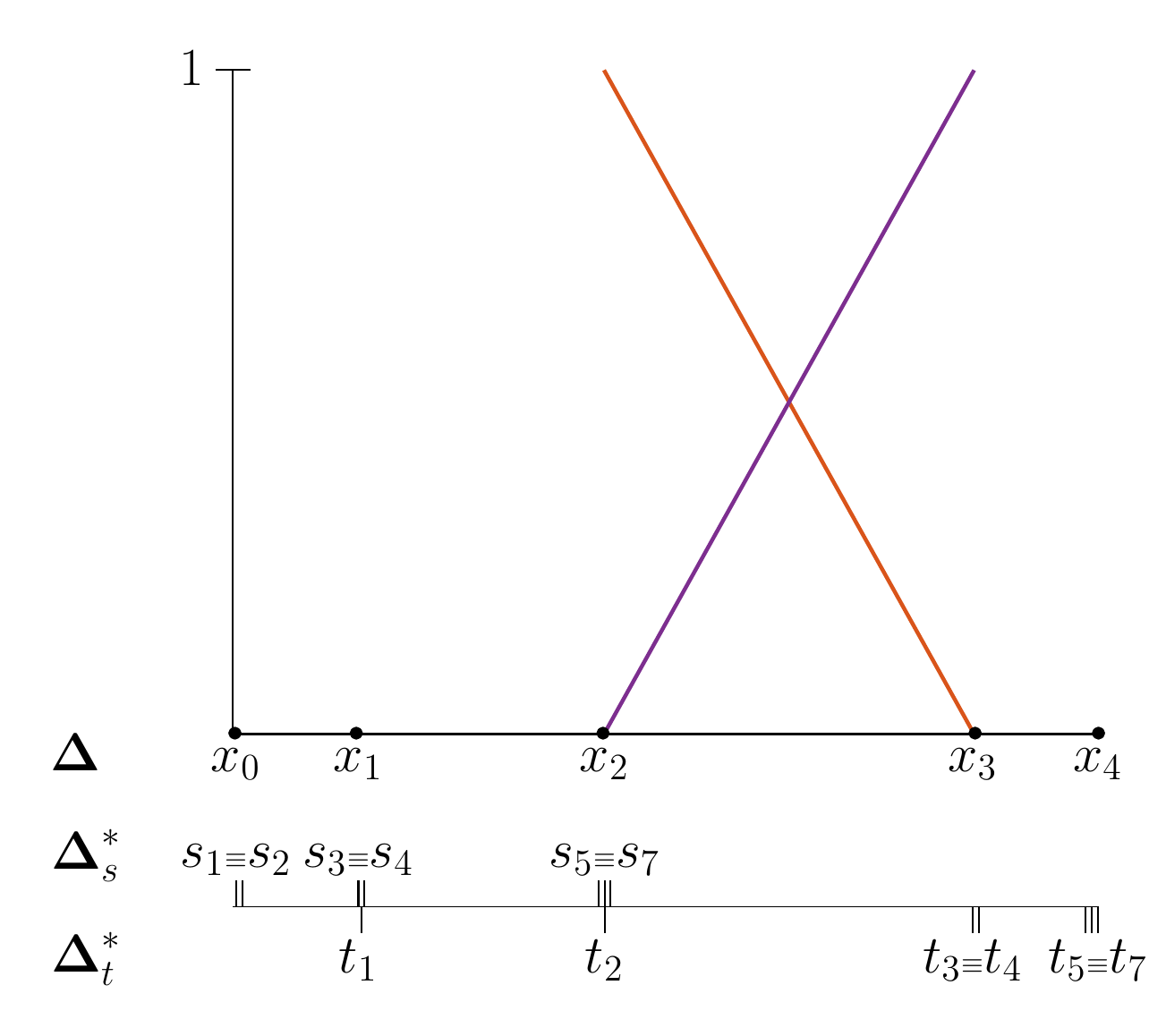} \vspace{0.0cm} &
\includegraphics[width=2.0in]{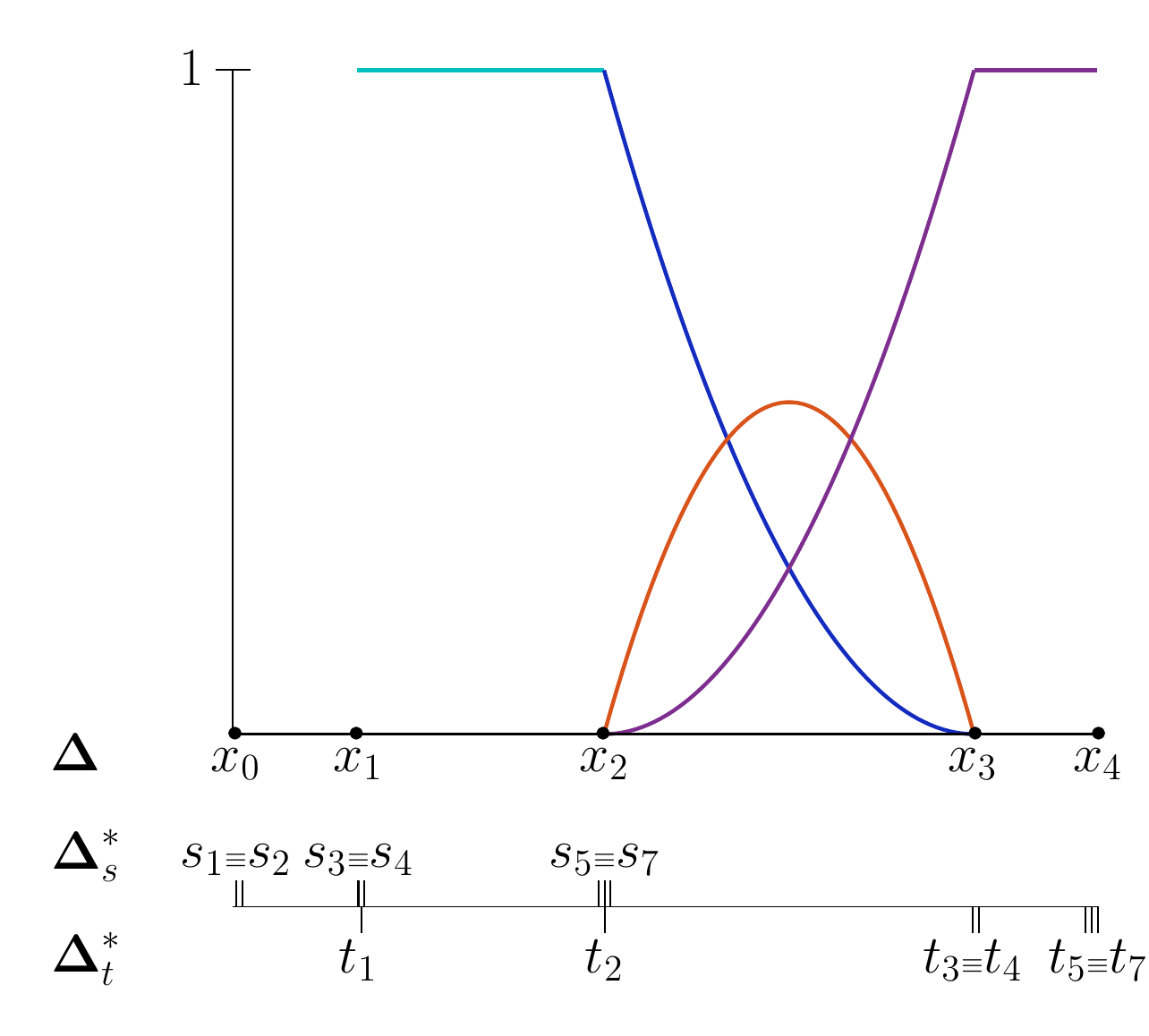} \vspace{0.0cm} \\
(a) $n = 0$ & (b) $n = 1$ & (c) $n = 2$
\end{tabular}\\
\begin{tabular}{cc}
\includegraphics[width=2.0in]{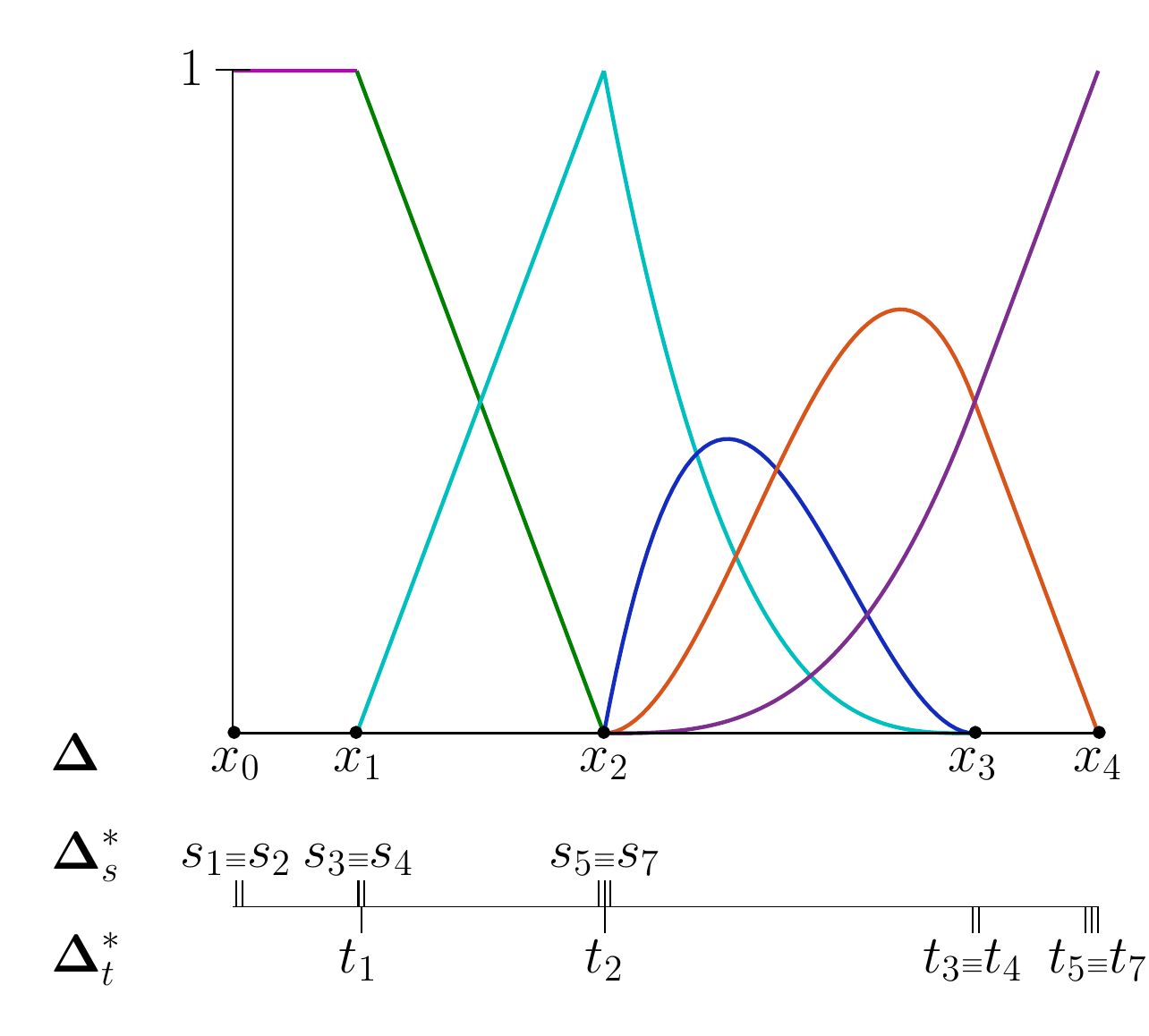} &
\includegraphics[width=2.0in]{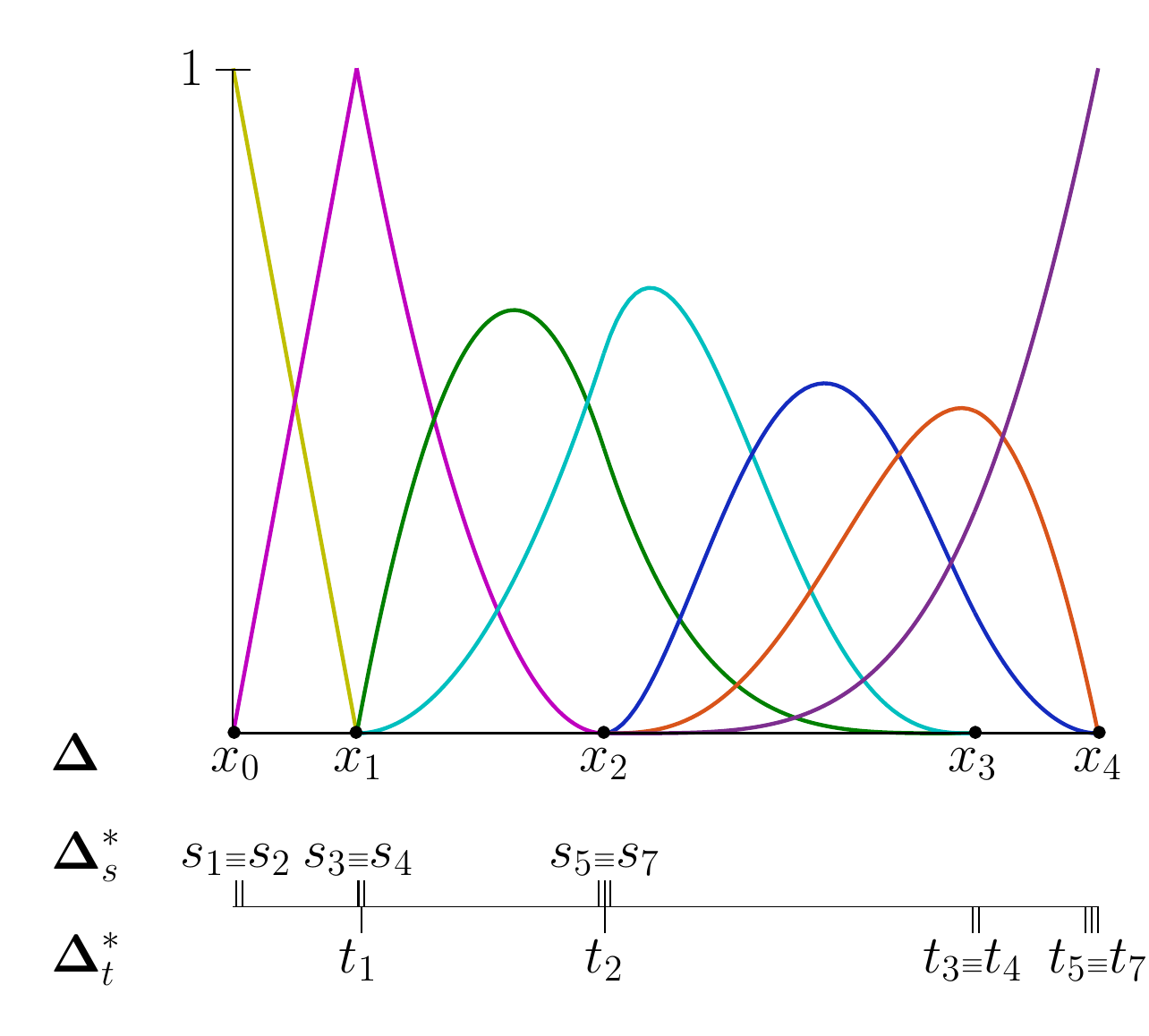} \vspace{0.0cm} \\
(d) $n = 3$ & (e) $n = 4$
\end{tabular}\\
\end{tabular}
\caption{Example \ref{ex:running}: Multi-degree B-spline basis $N_{i,n}$, for $n = 0, \dots,4$, and $i=5-n,\ldots,7$. Only the nonzero basis functions are displayed, namely (a) $N_{7,0}$; (b) $N_{i,1}$, $i=6,7$; (c) $N_{i,2}$, $i=4,\dots,7$; (d) $N_{i,3}$, $i=2,\dots,7$; (e) $N_{i,4}$, $i=1,\dots,7$.}
\label{fig:basisf}
\end{figure}

\begin{exmp}\label{ex:running}
Let us consider a MD-spline space $S(\PPP_{\mathbf{d}},\KKK,\bDelta)$ where the partition $\bDelta$ on the interval $[0,7]$
is given by $\bDelta \coloneqq \{x_1,x_2,x_3\}=\{1, 3, 6\}$, the polynomial segments
on each break-point interval have degrees $\mathbf{d} \coloneqq (d_0,\dots,d_3)=(1,2,4,2)$,
and the smoothness is defined by the vector $\mathcal{K} \coloneqq (k_1,k_2,k_3)=(0,1,2)$.
In Fig.\ \ref{fig:ex1_part}, the sequence of break-point intervals on $\bDelta$ and the associated extended knot
partitions $\bDelta_s^*\coloneqq\left\{s_j\right\}_{j=1}^{7}$ and
$\bDelta_t^*\coloneqq\left\{t_j\right\}_{j=1}^{7}$ are illustrated, together with
the degree of the polynomial segments and the continuity at the break-points.
The potentials of the proposed construction are highlighted by this simple example, which comprises the cases of maximum continuity and local reduction of continuity through multiple knots.
In particular, we impose a $C^0$ join at $x_1$ between the two consecutive sections with different degrees $d_0=1$ and $d_1=2$,
a $C^1$ join between the two sections of degrees $d_1=2$ and $d_2=4$,
and maximum continuity $C^{2}$ (corresponding to no knot in $\bDelta_s^*$)
is imposed in $x_3$.

The function sequence $\{N_{i,n}(x)\}$ built by the recursive process \eqref{def:int_rec} is illustrated
in Fig.\ \ref{fig:basisf} for increasing levels $n=0,\ldots,4$ which correspond to MD-spline spaces with increasing
maximum degree. Note that, by construction, the support of each
$N_{i,n}$ is $[s_i,t_{i-4+n}]$.
The final sequence $\{N_{i,4}(x)\}_{i=1}^{7}$, shown in Fig.\ \ref{fig:basisf} (e),
is the multi-degree B-spline basis of $S(\PPP_{\mathbf{d}},\KKK,\bDelta)$.
\end{exmp}

\section{Computation of the B-spline basis}
\label{sec:computation}

\subsection{Remarks on the existence of a Cox-de Boor type recurrence formula}
\label{sec:Cox-de Boor}
In the present section we investigate the possibility of evaluating MD-splines by means of a Cox-de Boor type recurrence relation
which is known to provide a stable and efficient algorithm for computing with polynomial splines.
To this aim, for suitably defined functions $\phi$'s, we shall consider the following recurrence relation
\begin{equation}
N_{i,n}(x)= \left \{
\begin{array}{ll}
1, \qquad x_j \leq x < x_{j+1} & n = m-d_j,\vspace{0.2cm}\\
\phi_i^{n-1}(x) N_{i,n-1}(x) + (1-\phi_{i+1}^{n-1}(x))N_{i+1,n-1}(x), & n>m-d_j, \vspace{0.2cm}\\
0,  & otherwise,
\end{array}
\right.
\label{eq:rec-form}
\end{equation}
for each $N_{i,n}$, for $n=0,\dots,m$ and $i=m+1-n,\dots,K$, where $m\coloneqq \max_i\{d_i\}$, defined on $[x_j,x_{j+1}) \subset [s_i,t_{i-m+n}]$ with $s_i < t_{i-m+n}$ (undefined $N_{i,n}$ functions shall be regarded as the zero function).
In this way, Cox-de Boor formula for the classical polynomial splines is a special instance of \eqref{eq:rec-form}.

For MD-splines one can exploit the integral recurrence relation \eqref{def:int_rec} to identify the symbolic expression for the  $\phi$'s.
Recalling that we have a clamped knot partition and considering the first B-spline function $N_{m+1-n,n}$ at level $n$, the following expression for $\phi_{m+2-n}^{n-1}$ is derived from \eqref{eq:rec-form}:
\begin{equation}\label{eq:phi}
\phi_{m+2-n}^{n-1}(x) = 1- \frac{N_{m+1-n,n}(x)}{{N_{m+2-n,n-1}(x)}}.
\end{equation}
By applying \eqref{eq:rec-form}, we can successively determine the remaining functions $\phi_i^{n-1}$, $i=m+3-n,\dots,K$.
We will use these relations to investigate whether the $\phi$'s
can be determined a priori without knowing the basis functions and whether they may have a simple form.
To this aim, we substitute in \eqref{eq:phi}, and in the other expressions similarly derived, the B-spline functions obtained through the integral procedure \eqref{def:int_rec} in order to get the $\phi$'s and see if their expression can be reduced to a simpler form.
We have verified that, in general, no simplified form for the $\phi$'s is found, as illustrated in the following example.

%

\begin{figure}
\centering
\subfigure[ $N_{i,0}$, $i=5$]{
{\includegraphics[width=0.3\textwidth]{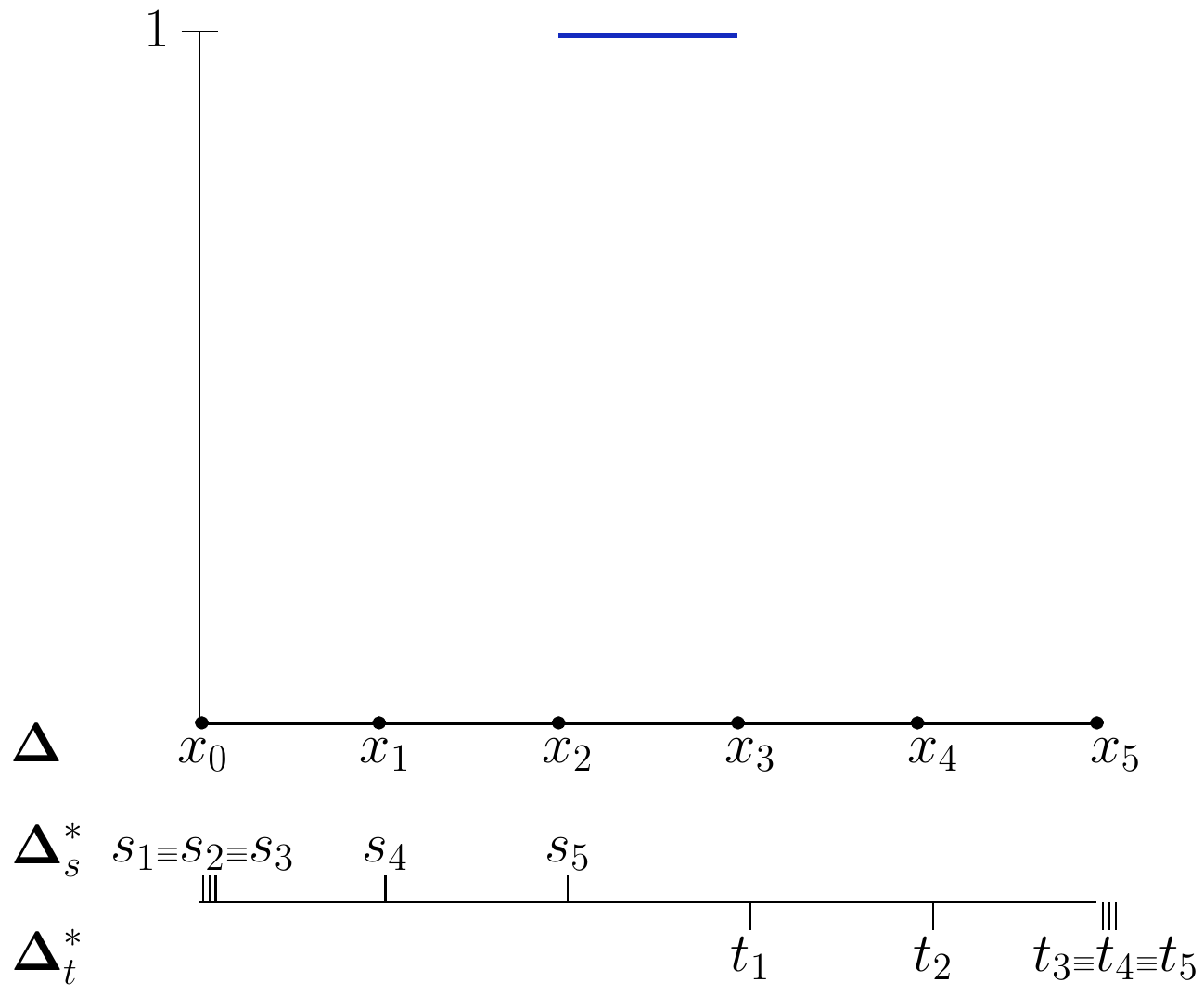}\label{fig:figN1}}
}
\subfigure[$\phi_i^0$, $i=4,5$]{
{\includegraphics[width=0.3\textwidth]{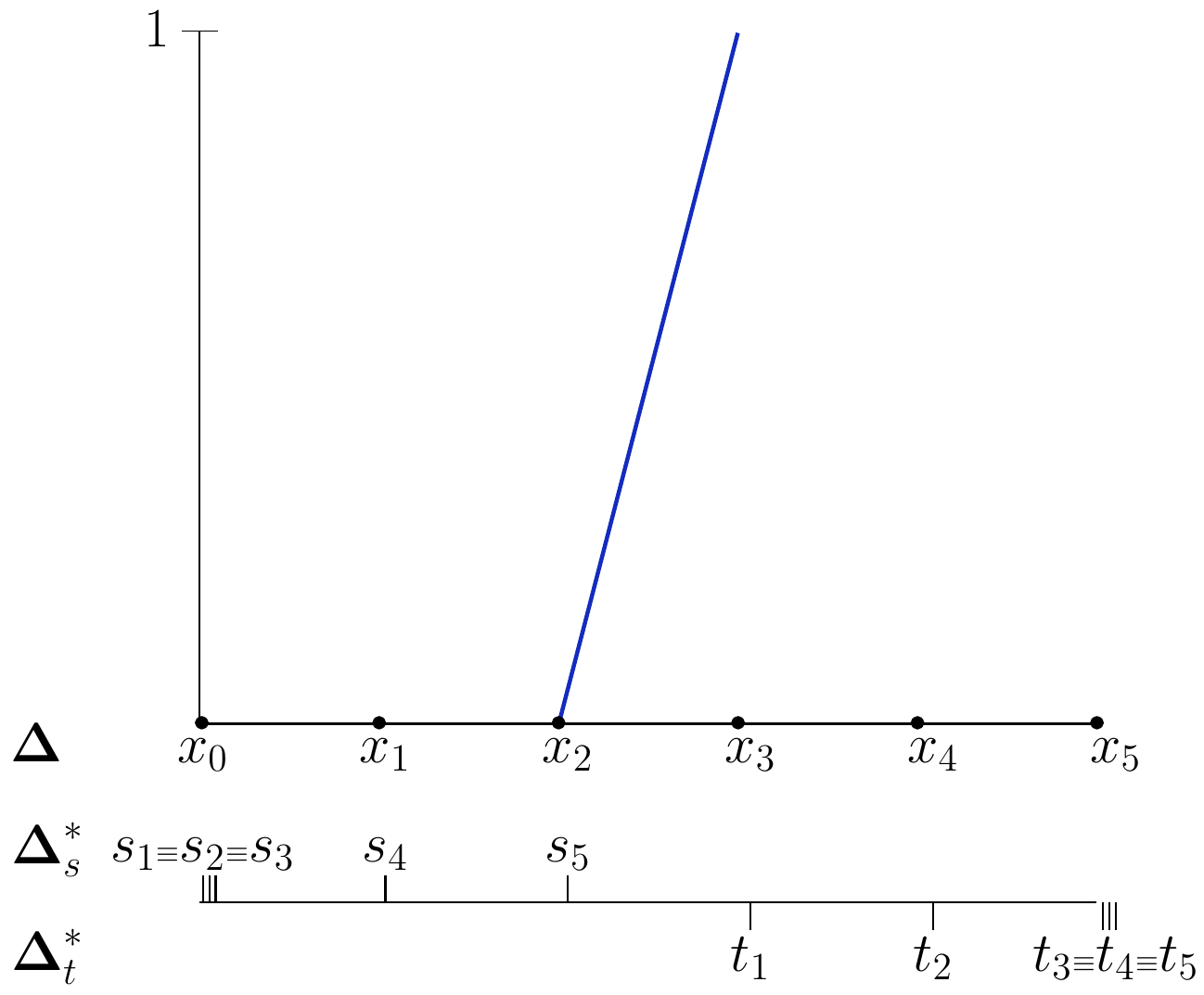}\label{fig:figPhi1}}
}
\subfigure[$N_{i,1}$, $i=4,5$]{
{\includegraphics[width=0.3\textwidth]{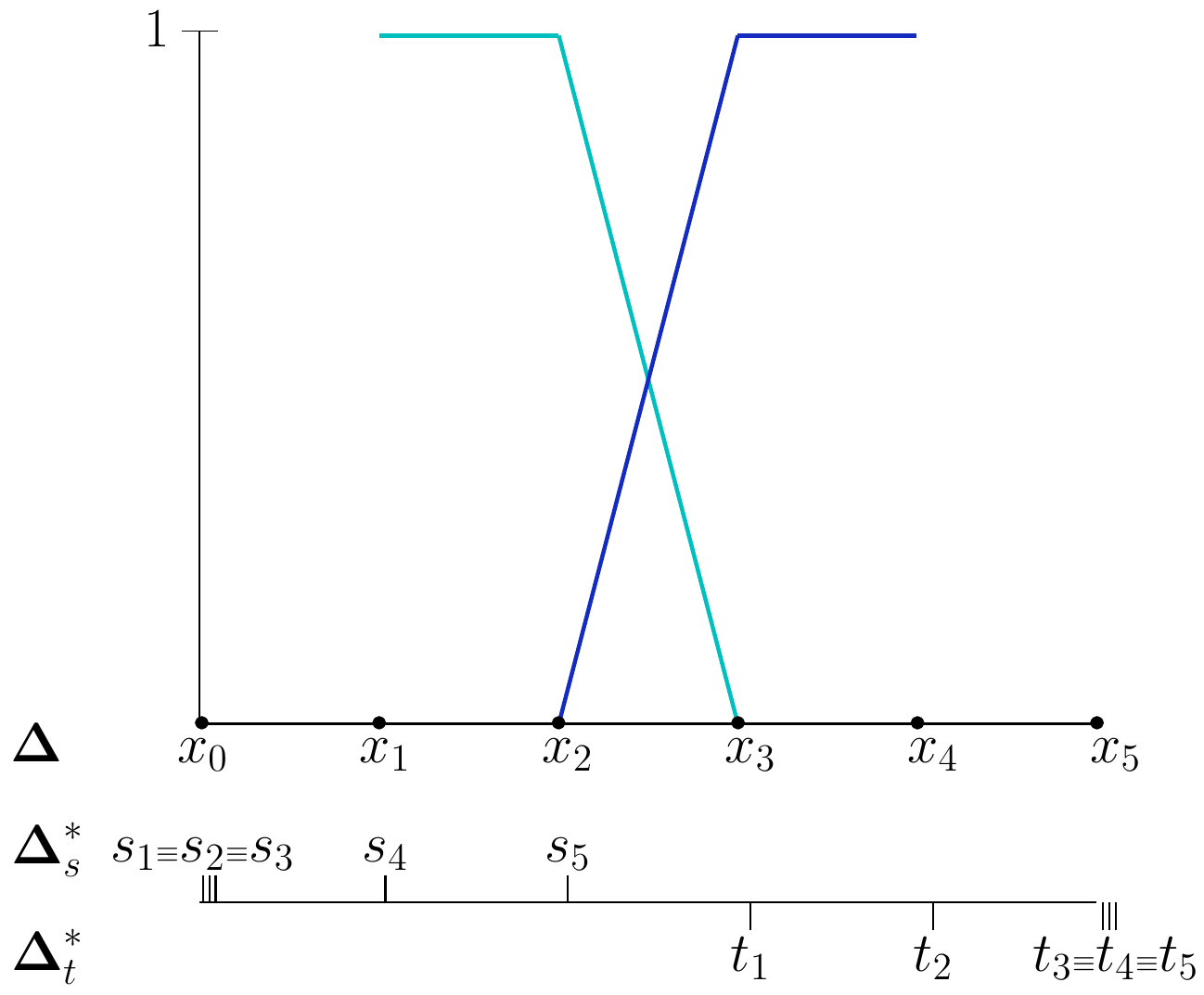}\label{fig:figN2}}
}
\centering
\subfigure[$N_{i,1}$, $i=4,5$]{
{\includegraphics[width=0.3\textwidth]{figures/ex_phi/MD_Ni1-crop}\label{fig:figN1}}
}
\subfigure[$\phi_i^1$, $i=3,4,5$]{
{\includegraphics[width=0.3\textwidth]{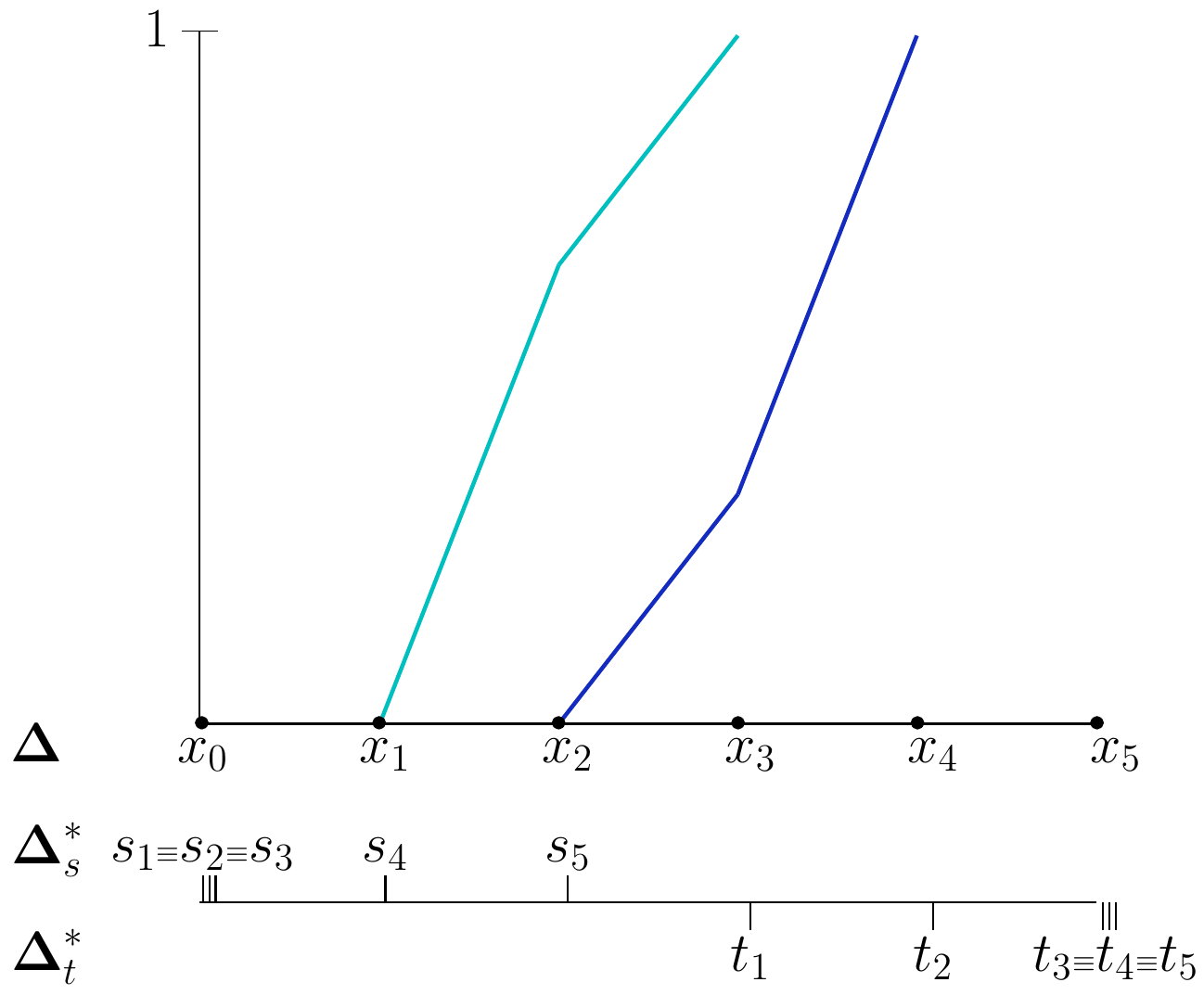}\label{fig:figPhi1}}
}
\subfigure[$N_{i,2}$, $i=3,4,5$]{
{\includegraphics[width=0.3\textwidth]{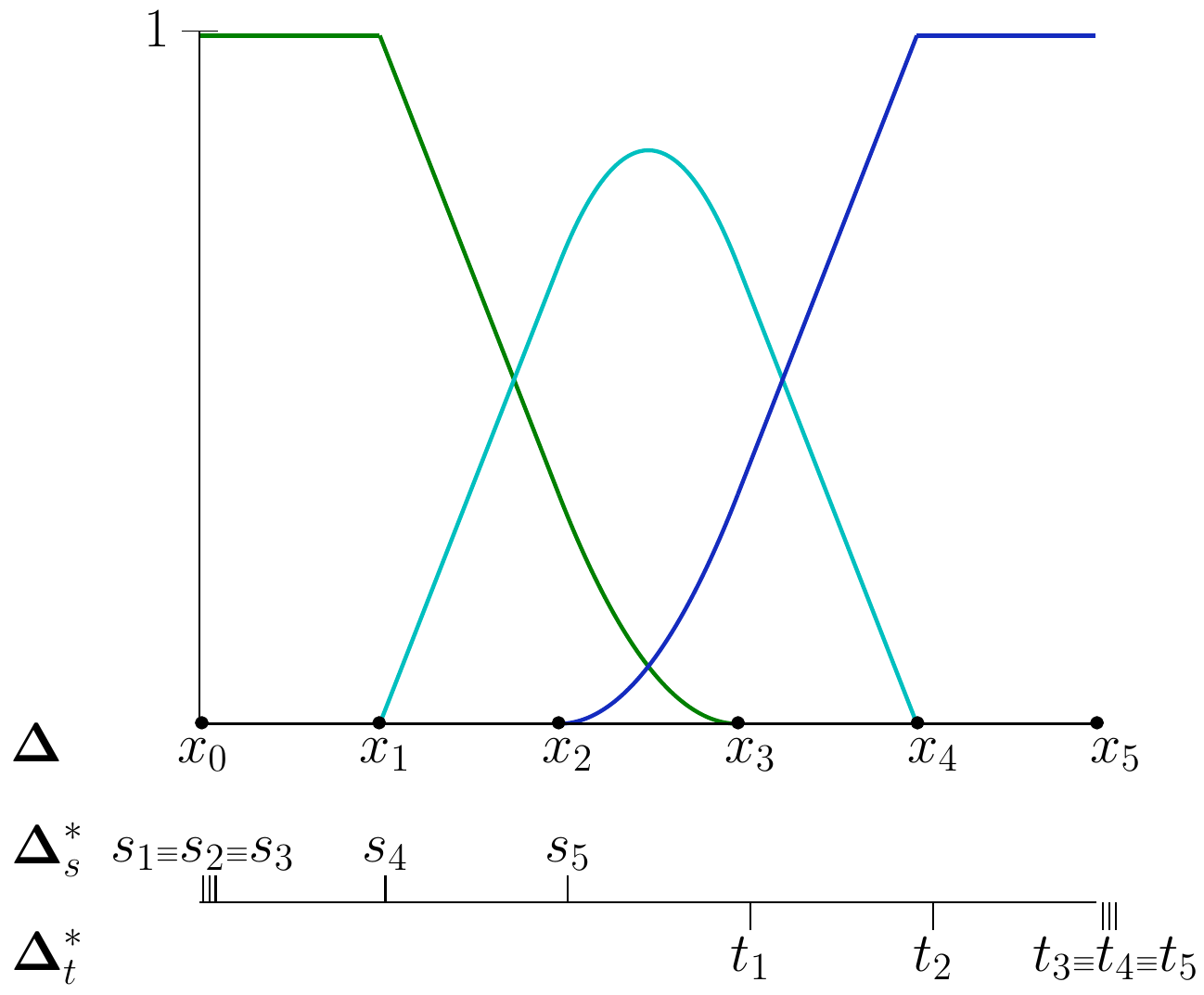}\label{fig:figN2}}
}
\subfigure[$N_{i,2}$, $i=3,4,5$]{
{\includegraphics[width=0.3\textwidth]{figures/ex_phi/MD_Ni2-crop}\label{fig:figN1}}
}
\subfigure[$\phi_i^2$, $i=2,\dots,5$]{
{\includegraphics[width=0.3\textwidth]{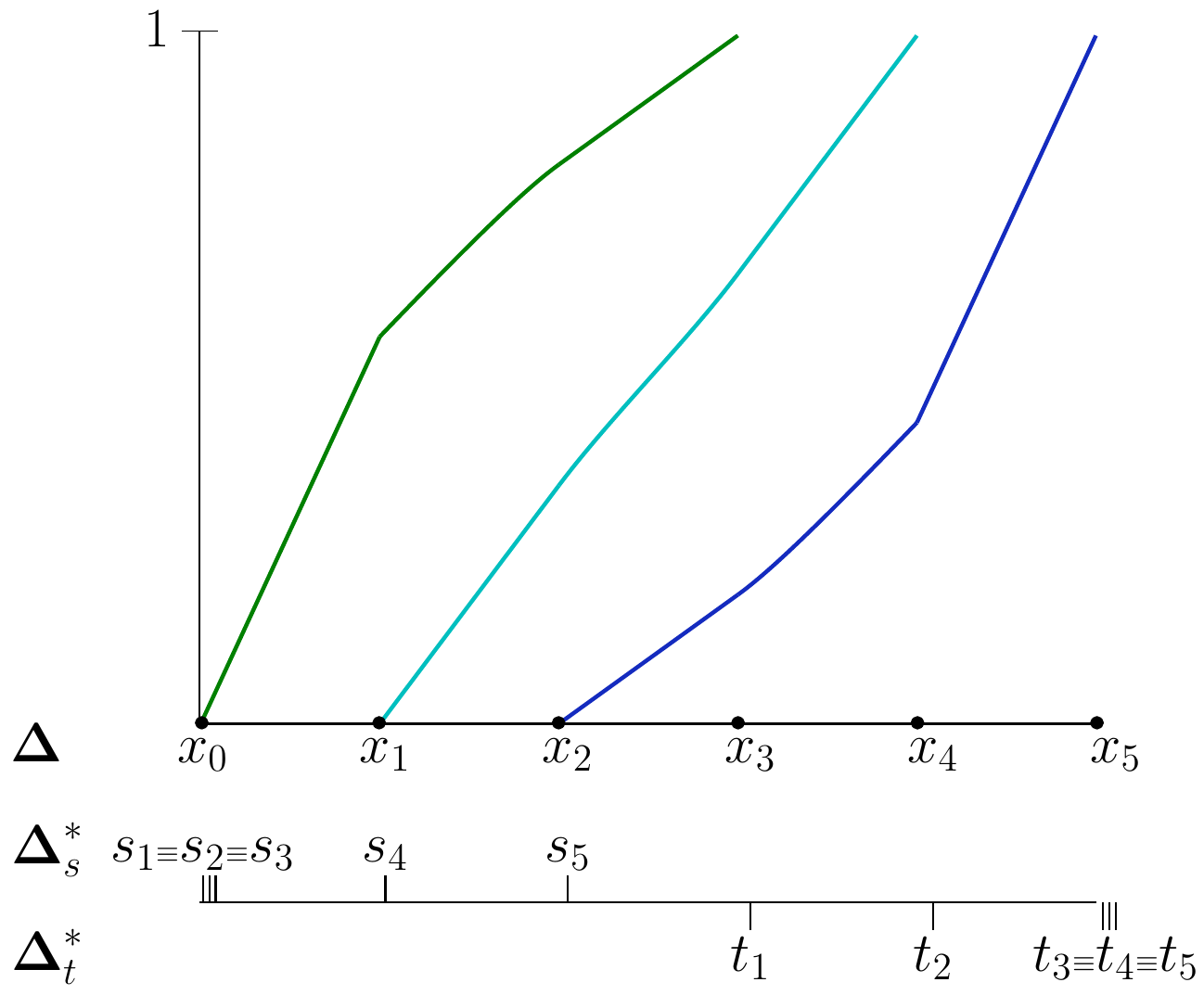}\label{fig:figPhi1}}
}
\subfigure[$N_{i,3}$, $i=2,\dots,5$]{
{\includegraphics[width=0.3\textwidth]{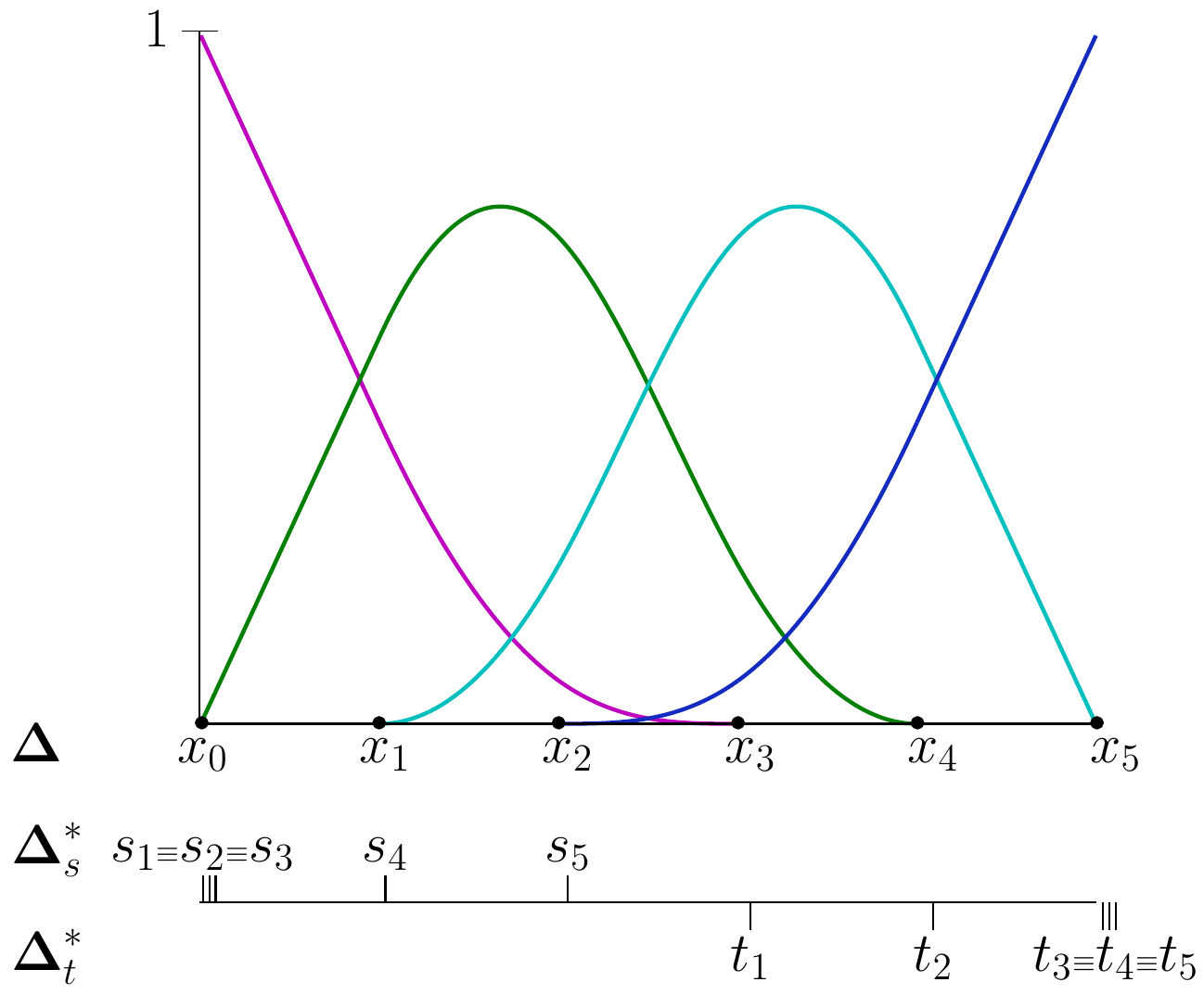}\label{fig:figN2}}
}
\subfigure[$N_{i,3}$, $i=2,\dots,5$]{
{\includegraphics[width=0.3\textwidth]{figures/ex_phi/MD_Ni3-crop}\label{fig:figN1}}
}
\subfigure[$\phi_i^3$, $i=1,\dots,5$]{
{\includegraphics[width=0.3\textwidth]{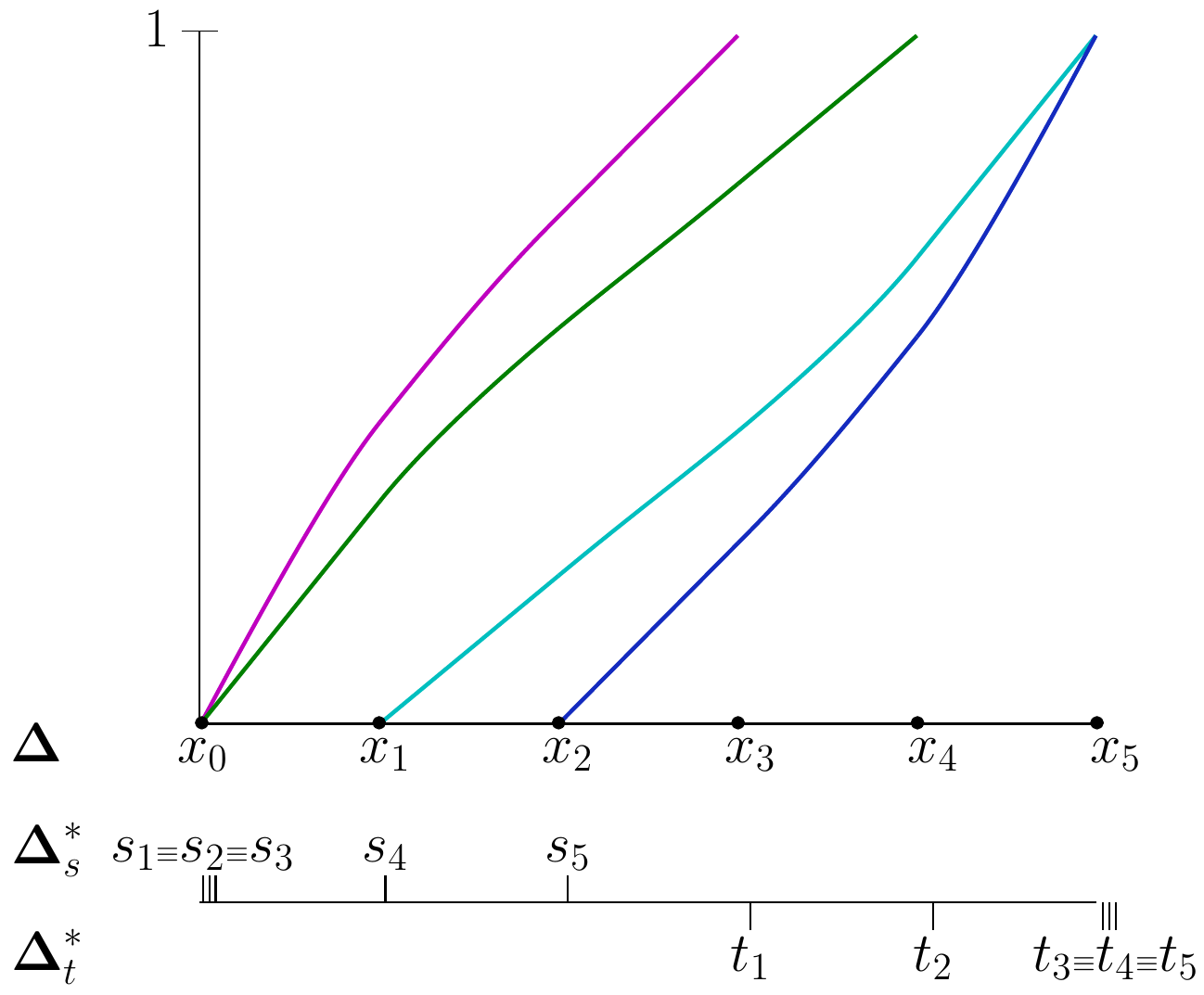}\label{fig:figPhi1}}
}
\subfigure[$N_{i,4}$, $i=1,\dots,5$]{
{\includegraphics[width=0.3\textwidth]{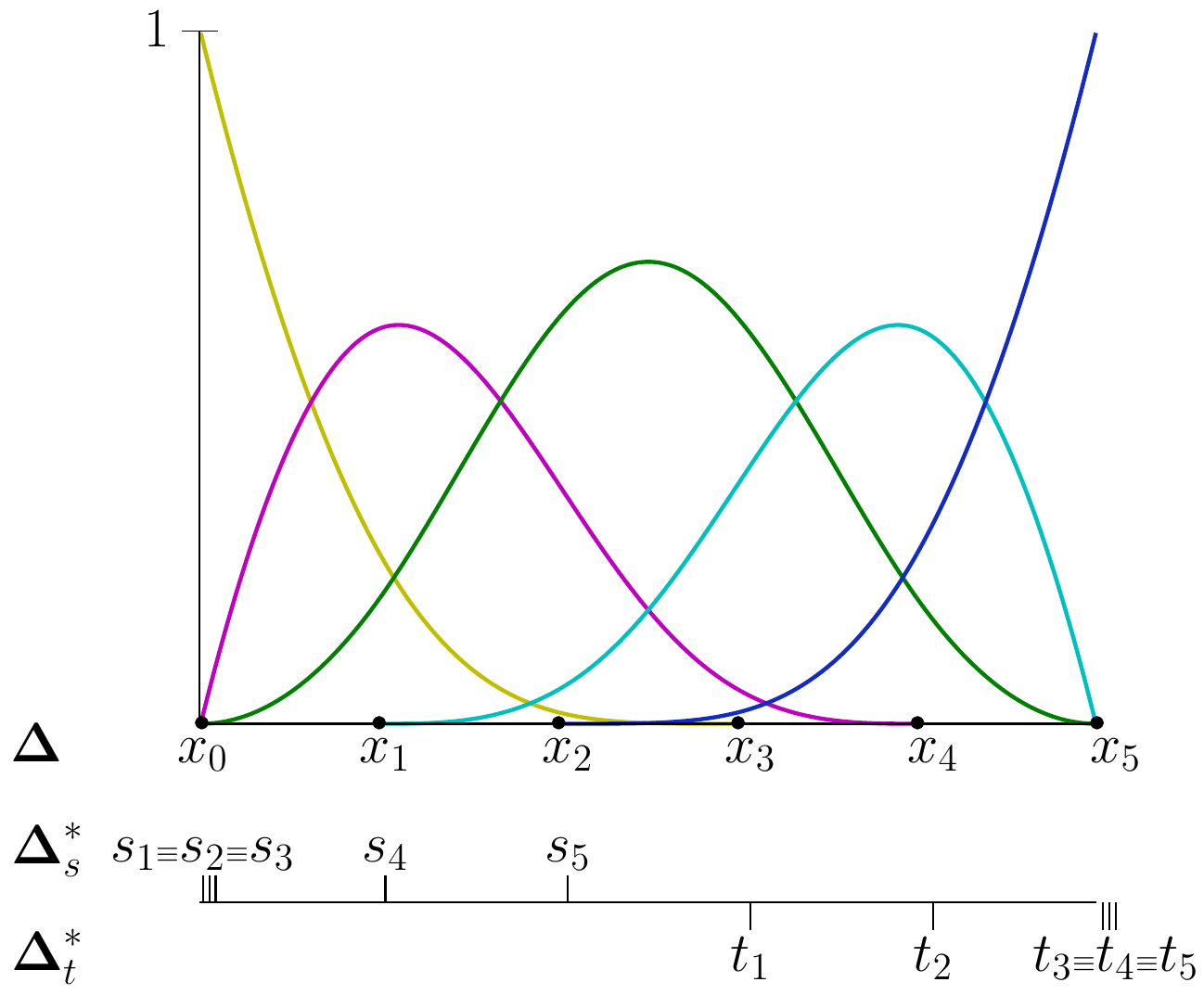}\label{fig:figN2}}
}
\caption{Example \ref{ex:ex2}: Basis functions $N_{i,n}$, $n=0,\dots,4$, and functions $\phi_i^n$, $n=0,\dots,3$.}
\label{fig_ex_phi}
\end{figure}
\begin{exmp}
\label{ex:ex2}
Let $S(\PPP_{\mathbf{d}},\KKK,\bDelta)$ be the space of MD-splines defined on the interval $[a,b]=[0,5]$, with break-point partition
$\bDelta=\{1,2,3,4\}$, degree vector $\mathbf{d}=(2,3,4,3,2)$, and continuities $\KKK=(2,3,3,2)$.
The third column of Fig.\ref{fig_ex_phi} reports the B-spline basis functions $N_{i,n}$ obtained by formula \eqref{eq:rec-form},
\ie by combining the functions $N_{i,n-1}$ in the first column with the functions $\phi_{i}^{n-1}$
in the second column, for each level $n=0,\ldots,4$.
Each $\phi_{i}^{n-1}$ is a piecewise functions defined in the interval $[s_i,t_{i-m+n-1}]$, as the associated $N_{i,n-1}$.
However, these functions are piecewise linear only for $n=0,1$ (first and second row in Fig.\ \ref{fig_ex_phi}), while for $n=2,3$ (third and fourth row in Fig.\ \ref{fig_ex_phi}) they assume
a piecewise rational form of higher degree. For example, the two functions $\phi_i^3$, $i=2,3$, starting at $x_0=0$ are defined as follows
\[
\phi_2^3(x)=
\begin{cases}
\scalebox{1.2}{$\frac{288x^2-475x}{61(9x-16)}$}, & 0 \leq x \leq 1,\\[1.3ex]
\scalebox{1.2}{$\frac{6x^3-297x^2+301x+119}{61(3x^2-15x+19)}$}, & 1 \leq x \leq 2,\\[1.3ex]
\scalebox{1.2}{$\frac{16x+13}{61}$}, & 2 \leq x \leq 3,
\end{cases}  \qquad
\phi_3^3(x)=
\begin{cases}
\scalebox{1.2}{$\frac{32x}{99}$}, & 0 \leq x \leq 1,\\[1.3ex]
\scalebox{1.2}{$\frac{16(-116x^3+285x^2+246x-181)}{297(29x^2-97x+29)}$}, & 1 \leq x \leq 2,\\[1.3ex]
\scalebox{1.2}{$\frac{16(45x^4-278x^3+177x^2+1182x-1045)}{297(15x^3-119x^2+277x-149)}$}, & 2 \leq x \leq 3,\\[1.3ex]
\scalebox{1.2}{$\frac{64x+41}{297}$}, & 3 \leq x \leq 4.\\
\end{cases}
\]
For instance, direct verification shows that the numerator of $\phi_2^3(x)$, for $0 \leq x \leq 1$, is precisely the difference of the level $n-1$ and $n$ basis functions $N_{2,3}$ and $N_{1,4}$, whereas the denominator is the B-spline basis of level $n-1$, \ie $N_{2,3}$.
More generally, it can be observed that
the numerator and denominator of the $\phi^{n-1}$ functions may have the same degrees as those of the elements of the B-spline basis that they define at level $n$ by means of \eqref{eq:rec-form} (see, \eg the second piece of $\phi_2^3$ and the second and third piece of $\phi_3^3$).
This simple experiment suggests that the level $n$ basis cannot be obtained, in general, as a combination of $\phi$'s having a degree which is lower than that of the target basis.
\end{exmp}

The previous example leads us to conclude that Cox-de Boor recurrence relation cannot promptly be generalized in order to work with MD-splines.
This is motivated by the observation that, at least in general, the $\phi$'s do not have a low-degree \qtext{simple} form, nor their expression can be determined a priori.
However, in some special cases,  explicit and easily computable formul\ae \ for the $\phi$'s can be found.

In particular, let us consider the space  $S(\PPP_{\mathbf{d}},\KKK,\bDelta)$ of MD-splines that are required to be $C^k$, $k=0,1$,
between two segments of different degrees and $C^{k}$, $0\leq k\leq d-1$ at the join of two segments of the same degree $d$.
In this case, we devise a simple formula such that each basis function $N_{i,n}$ can be evaluated by means of \eqref{eq:rec-form} for suitable $\phi$'s.
Algorithm 1 provides the evaluation procedure. We omit the related proof, as discussing these details goes
beyond the scope of this paper.
It is interesting to notice, however, that the functions $\phi^{n-1}$
are linear for $n=0,\ldots, m-1$, while only at the last level $n=m$ they assume a piecewise linear form
on an interval defined by at most 3 break-point intervals.

\resizebox{.93\textwidth}{!}{%
\begin{minipage}{\textwidth}
\begin{algorithm}[H]
\begin{flushleft}
\begin{itemize}[leftmargin=*]
\item[--] For $n=1\dots,m-1$, compute $N_{i,n}$, $i=m+1-n,\dots,K$, through \eqref{eq:rec-form} with
\begin{equation}\label{eq:phi_special}
\phi_i^{n-1}=\frac{x-s_i}{t_{i-m+n-1}-s_i}, \quad i=m+2-n,\dots,K.
\end{equation}
\item[--] For $n=m$,  compute $N_{i,m}$, $i=1,\dots,K$, by \eqref{eq:rec-form} with the $\phi_i^{m-1}$ computed as follows
\begin{itemize}
\item  Compute the number of break-point intervals $n_i$ in $\supp N_{i,m-1} = [s_i,t_{i-1}]$
\ie $n_i= pt_{i-1}-ps_i$.
\item  If $n_i=2$ or $n_i=3$, and the degrees $d_{j}$, $j=ps_i,\dots,pt_{i-1}-1$ are not equal, then
\vspace{-0.4cm}
\begin{equation}\label{eq:philin}
\phi_i^{m-1}(x)=
\left\{ \begin{array}{ll}
\displaystyle{\frac{1}{\delta}\frac{x-x_{ps_i}}{d_{ps_i}}}, & x\in[x_{ps_i},x_{{ps_i}+1}], \vspace{0.2cm}\\
\displaystyle{\frac{1}{\delta}\left(\frac{x_{ps_i+1}-x_{ps_i}}{d_{ps_i}}+\frac{x-x_{ps_i+1}}{d_{ps_i+1}}\right)}, & x\in[x_{ps_i + 1},x_{ps_i+2}], \vspace{0.2cm}\\
\displaystyle{\frac{1}{\delta}\left(\frac{x_{ps_i+1}-x_{ps_i}}{d_{ps_i}}
+\frac{x_{ps_i+2}-x_{ps_i+1}}{d_{ps_i+1}}+\frac{x-x_{ps_i+2}}{d_{ps_i+2}}\right)}, & x\in[x_{ps_i+2},x_{pt_{i-1}}], \vspace{0.2cm}\\
\end{array} \right.
\quad \;\;
\displaystyle{\delta\coloneqq \sum_{j=ps_i}^{pt_{i-1}-1} \frac{x_{j+1}-x_{j}}{d_j}}.
\end{equation}

[If $n_i=2$ only the first two equations are applied].
\item Otherwise (same degrees $d_j$ or $n_i=1$ and $n_i>3$) compute $\phi_i^{m-1}(x)$ using relation \eqref{eq:phi_special}.
\vspace{-0.2cm}
\end{itemize}
\end{itemize}
\end{flushleft}
\caption{}
\label{alg:algor1}
\end{algorithm}
\end{minipage}
}

\clearpage
\begin{rem}
In \cite{Li2012} the authors proposed a Cox-de Boor type recurrence formula for MD-splines characterized by exact $C^1$
continuity between two adjacent segments with different degrees, and exact $C^{d-1}$ continuity between two adjacent segments
of same degree $d$.
Our proposal for this specific case, formulated in Algorithm 1,
differs from the evaluation method developed in \cite{Li2012}. In addition, the procedure in Algorithm 1 allows for
attaining any order of continuity $C^r$, $r=0,\dots,d-1$, between adjacent segments of the same degree and $C^r$, $r=0,1$, between segments of different degree.
\label{sec:special_case}
\end{rem}

\subsection{Transition functions for multi-degree splines}
\label{sec:trans_fun}

In this section we introduce the transition functions as a tool for the computation of the B-spline basis. Their name takes up the terminology of previous papers, dealing with splines having sections all of the same dimension, belonging to either polynomial spaces \cite{ABC2013a,BCR2013a} or to the wider class of extended Chebyshev spaces \cite{BCR2017}.

\begin{defn}[Transition functions]\label{def:trans_function}
Let $S(\PPP_{\mathbf{d}},\KKK,\bDelta)$ be a multi-degree spline space of dimension $K$, with $\bDelta=\{x_i\}_{i=1}^{q}$
a partition of $[a,b]$, $\bDelta_s^*=\{s_i\}_{i=1}^{K}$ and  $\bDelta_t^*=\{t_i\}_{i=1}^{K}$ the associated left and right extended partitions, and with B-spline basis
$\left\{N_{i,m}\right\}_{i=1}^{K}$.
The associated \emph{transition functions} $\left\{f_i\right\}_{i=1}^K$ are given by:
\begin{equation}\label{eq:f_i_N_i}
 f_{i}\coloneqq \sum_{j=i}^{K} N_{j,m}, \quad i=1,\dots,K.
\end{equation}
\end{defn}

Assuming $f_{K+1}\equiv 0$, we can reverse relation \eqref{eq:f_i_N_i} expressing the B-spline basis in terms of transition functions as
\begin{equation}\label{eq:N}
N_{i,m} = f_{i}-f_{i+1}, \quad i=1,\dots,K.
\end{equation}

With reference to the end point property of the B-spline basis (see Definition \ref{def:B-spline basis}), we introduce the quantities
\begin{equation}\label{eq:KisKit}
k_i^s\coloneqq d_{ps_i}-\max \{ j\geq 0 \ | \ s_i=s_{i+j} \}-1 \quad \textrm{and} \quad k_i^t\coloneqq d_{pt_i-1}-\max \{ j \geq 0 \ | \ t_{i-j}=t_i \}-1.
\end{equation}
Hence, from the properties of the B-spline basis we can deduce that $f_{1}(x) = 1$, for all $x \in [a,b]$,
and that the piecewise functions $f_{i}$, $i=2,\dots,K$, have the following characteristics:
\begin{enumerate}[label=\alph*)]
\item \label{propty:f1} 
$f_i$ is nonnegative and
\begin{equation}
f_{i}(x)=
\begin{cases}
0, & x \leq s_i,\\
1, & x \geq t_{i-1};
\end{cases}
\end{equation}

\item \label{propty:f2} 
$f_{i}$ vanishes exactly $k_i^s+1$ times at $s_i$ and $1-f_{i}$ vanishes exactly $k_{i-1}^t+1$ times at $t_{i-1}$;

\item \label{propty:Tayor}  the Taylor expansions of $f_i$ at $s_i$ and $t_{i-1}$ show that
\begin{equation}\label{eq:Taylor}
D_+^{k_{i}^s+1}f_{i}(s_{i}) > 0 \qquad
\text{and} \qquad
(-1)^{k_{i-1}^t+1}D_-^{k_{i-1}^t+1}f_{i}(t_{i-1}) > 0.
\end{equation}
\end{enumerate}

Property a) entails that $f_{i}$, $i=2,\dots,K$, is nontrivial (i.e. it is neither the constant function zero or one) in the interval $[s_i,t_{i-1}]$ only. In particular,
denoted by $x_{ps_i},\dots, x_{pt_{i-1}}$ the break-points of $\bDelta$ contained in $[s_i,t_{i-1}]$,
the continuity conditions at the break-points along with property b) yield the following relations:
\begin{equation}
\begin{alignedat}{3}
& D^r_+ f_{i}(x_{ps_i})=0, &\qquad& r=0,\dots,k_{i}^s,\\
& D^r_- f_{i}(x_{j})=D^r_+ f_{i}(x_{j}), &\qquad& r=0,\dots,k_{j}, \quad j=ps_i+1,\dots,pt_{i-1}-1, \\\label{eq:cond_tf}
& D^r_- f_{i}(x_{pt_{i-1}})=\delta_{r,0}, &\qquad& r=0,\dots,k_{i-1}^t.
\end{alignedat}
\end{equation}

\smallskip
In Proposition \ref{prop:unique} we show that each transition function $f_i$, $i=2,\dots,K$, can simply be computed by solving the linear system represented by the conditions \eqref{eq:cond_tf}.
This prompts us to construct the B-spline basis by first computing the transition functions through \eqref{eq:cond_tf} and subsequently applying \eqref{eq:N}.
By way of illustration, the transition functions for the MD-spline space considered in Example \ref{ex:running} are depicted in Figure \ref{fig:tf}. Their combinations, according to \eqref{eq:N}, yield the B-spline basis functions in Fig.\ \ref{fig:basisf} (e).

\begin{figure}[tbh]
\centering
\includegraphics[width=3.0in]{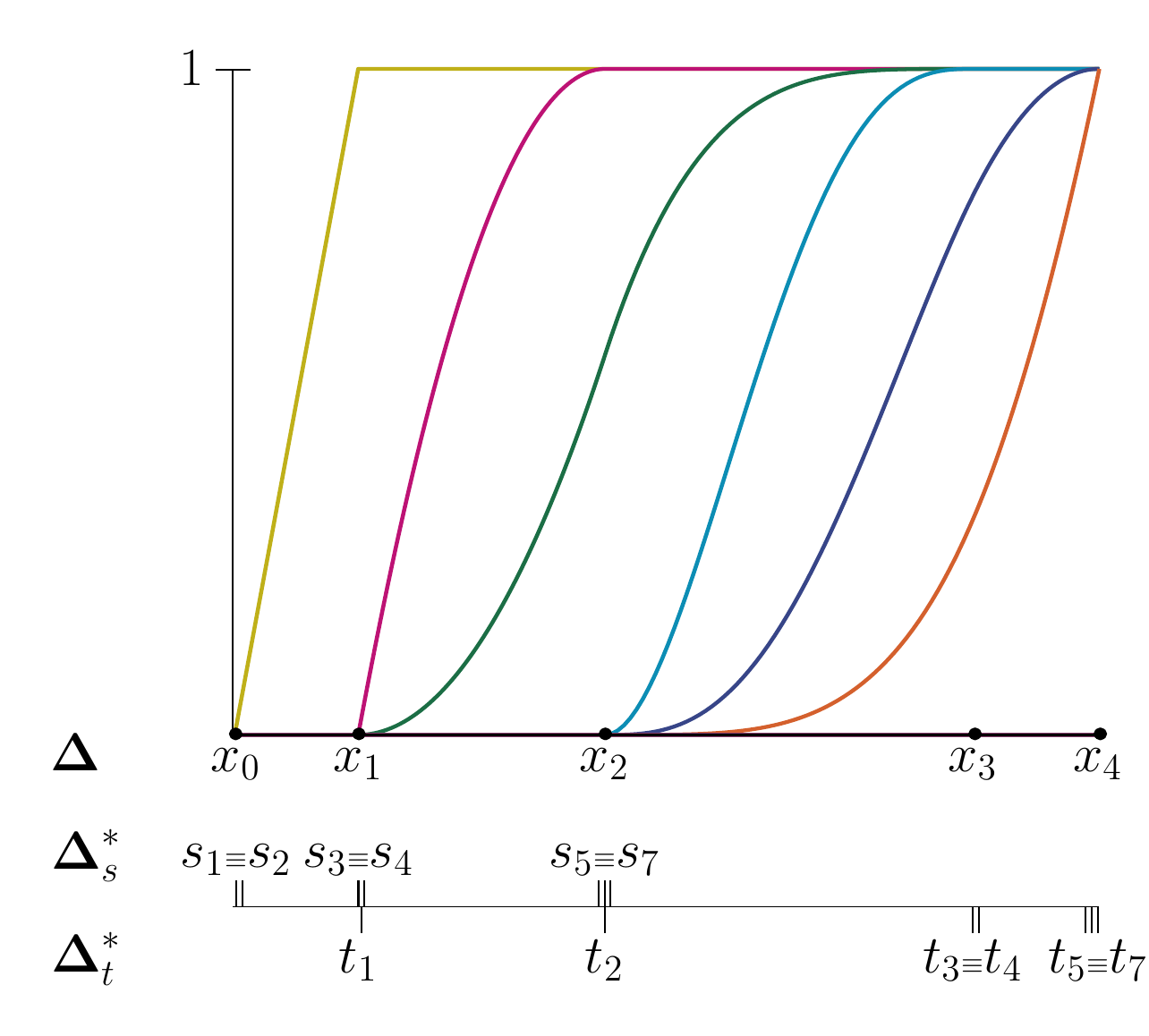}
\caption{Transition functions $f_i$, $i=2,\ldots,7$, for the MD-spline space in Example \ref{ex:running}.
The related B-spline basis is depicted in Fig.\ \ref{fig:basisf} (e).}
\label{fig:tf}
\end{figure}

\begin{prop}\label{prop:unique}
Each transition function $f_i$ $i=2,\dots,K$, in the interval $[s_i,t_{i-1}]$, is uniquely determined by conditions \eqref{eq:cond_tf}.
\end{prop}

\begin{proof}
The proof consists in showing that conditions \eqref{eq:cond_tf} give rise to a linear system with a square, non-singular matrix.

We start by verifying that the number of endpoint conditions b) equals the dimension of the restriction of $S(\PPP_{\mathbf{d}},\KKK,\bDelta)$ to $[s_i,t_{i-1}]$, which we denote by $S_{[s_i,t_{i-1}]}(\PPP_{\mathbf{d}},\KKK,\bDelta)$, and thus the system matrix is square.
As a consequence of the definition of $\bDelta_s^*$ and $\bDelta_t^*$, the value of the index $i$ associated with $s_i$ and $t_{i-1}$ is given by
\[
i = d_0 + \sum_{j=1}^{ps_i-1}(d_j-k_{j}) + k_{i}^s - k_{ps_i} + 2
= \sum_{j=1}^{pt_{i-1}-1}(d_{j-1}-k_{j}) + d_{pt_{i-1}-1} - k_{i-1}^t + 1.
\]
The above identities yield
\[
k_i^s+1 = i - d_0 - \sum_{j=1}^{ps_i-1}(d_j-k_{j}) + k_{ps_i} -1, \]
and
\[k_{i-1}^t+1 = -i+ \sum_{j=1}^{pt_{i-1}-1}(d_{j-1}-k_{j}) + d_{pt_{i-1}-1} +2.
\]
Adding up the last two equalities we get
\begin{equation*}
\begin{aligned}
k_i^s+1 + k_{i-1}^t+1 =& - \sum_{j=1}^{ps_i-1}(d_{j-1}-k_{j}) - d_{ps_i-1} + k_{ps_i} \\
&+ \sum_{j=1}^{pt_{i-1}-1}(d_{j-1}-k_{j}) + d_{pt_{i-1}-1}+1\\
=&\sum_{j=ps_i}^{pt_{i-1}-1}(d_{j-1}-k_{j}) - d_{ps_i-1} + k_{ps_i} + d_{pt_{i-1}-1}+1\\
=&\sum_{j=ps_i+1}^{pt_{i-1}-1}(d_{j-1}-k_{j}) + d_{pt_{i-1}-1} +1,
\end{aligned}
\end{equation*}
which proves that the number of endpoint conditions equals the dimension of $S_{[s_i,t_{i-1}]}(\PPP_{\mathbf{d}},\KKK,\bDelta)$.
\\
According to \cite{BMuhl2003}, which generalizes the classical results on spline interpolation to splines with sections of different dimensions,
to guarantee that the system matrix is nonsingular, we shall verify that the interpolation nodes $x_{ps_i}$ and $x_{pt_{i-1}}$ and the knots of $S_{[s_i,t_{i-1}]}(\PPP_{\mathbf{d}},\KKK,\bDelta)$ satisfy proper interlacing conditions and use the zero bound \eqref{eq:zeros}.
Denoted by $K_i$ the dimension of $S_{[s_i,t_{i-1}]}(\PPP_{\mathbf{d}},\KKK,\bDelta)$, by $I_j$, $j=1,\ldots,K_i$, the support of the $j$th B-spline function in $S_{[s_i,t_{i-1}]}(\PPP_{\mathbf{d}},\KKK,\bDelta)$ and by
$\tau_j$, $j=1,\ldots,K_i$, the nodes, the interlacing conditions amount to requiring that $\tau_j \in I_j$, $j=1,\ldots,K_i$ \cite{BMuhl2003}.
In our situation we have $\tau_j=x_{ps_i}$, $j=1,\ldots,k_i^s+1$, and $\tau_j=x_{pt_{i-1}}$, $j=K_i-k_{i-1}^t,\ldots,K_i$.
Since by construction $k_i^s \leq d_{ps_i}$ and $k_{i-1}^t \leq d_{pt_{i-1}-1}$, the interlacing conditions are always satisfied.
In fact $x_{ps_i}$ belongs to the supports $I_j$, $j=1,\ldots,d_{ps_i}+1$, while
$x_{pt_{i-1}}$ belongs to $I_j$, $j=K_i-d_{pt_{i-1}-1},\dots,K_i$.
\end{proof}

In the remainder of this section we delve on practical aspects concerned with the calculation of the transition functions.
For computational purposes it is convenient to rely on the Bernstein bases associated with the individual section spaces.
Namely, let $\{B_{h,d_j}\}_{h=0}^{d_j}$ be the degree-$d_j$ Bernstein basis on $[x_j,x_{j+1}]$.
Then the transition function $f_i$, which is nontrivial in $[s_i,t_{i-1}]$, can be expressed as follows
\begin{equation}\label{eq:expansion}
f_{i}(x)=\sum_{h=0}^{d_j} b_{i,j,h}B_{h,d_j}(x), \quad x\in [x_j,x_{j+1}], \quad  j=ps_i,\dots,pt_{i-1}-1,
\end{equation}
where $b_{i,j,h}$, $h=0,\dots,d_j$, are the coefficients of the
local expansions of $f_{i}$ on $[x_j,x_{j+1}]$.

According to the first and last row of \eqref{eq:cond_tf}, the first $k_i^s+1$ coefficients of the local expansion of the first piece of $f_i$ will be $0$,
while, the last $k_{i-1}^t+1$ coefficients of the local expansion of the last piece of $f_i$ will be equal to $1$.
These conditions fully determine those transition functions which are nontrivial on one interval only.
As for the others, the undetermined coefficients in \eqref{eq:expansion} can be computed by solving the linear system
given by the second row of \eqref{eq:cond_tf}.

Having expressed the transition functions in the local Bernstein bases, we can easily compute their integrals and derivatives by means of the standard relations
\[
D B_{h,d_j}(x)=\frac{d_j}{x_{j+1}-x_j}\left(B_{h-1,d_j-1}(x)-B_{h,d_j-1}(x)\right)\qquad \mathrm{and} \qquad
\int_{x_j}^{x_{j+1}} B_{h,d_j}(x) \, dx = \frac{x_{j+1}-x_j}{d_j+1}, \quad h=0,\dots,d_j.
\]
Moreover, in view of \eqref{eq:N}, application of the above formul\ae \ promptly allows for computing derivatives and integrals of the B-spline basis.

Finally we remark that the proposed approach automatically yields the relation between the B-spline basis
and the local Bernstein bases. More precisely,
the restriction of any spline $f\in S(\PPP_{\mathbf{d}},\KKK,\bDelta)$  to the interval $[x_j,x_{j+1}]$, $s_\ell \leq x_j < \min\left( s_{\ell+1},b\right)$, can be written
using relations \eqref{eq:spline_local}, \eqref{eq:N} and \eqref{eq:expansion} as follows:
\[f(x)= \sum_{i=\ell-d_j}^{\ell} c_i\,N_{i,m}(x) = \sum_{i=\ell-d_{j}}^{\ell} c_i\,\left(f_{i}(x)-f_{i+1}(x)\right) = \sum_{h=0}^{d_j}\tilde{b}_{h,j} B_{h,d_j}(x),\]
where
\[
\tilde{b}_{h,j}\coloneqq \sum_{i=\ell-d_{j}}^{\ell} c_i\left(b_{i,j,h}-b_{i+1,j,h}\right).
\]
Therefore the $\tilde{b}_{h,j}$'s are the coefficients of the local expansion of  $f$ in the Bernstein basis of degree $d_j$ over $[x_j,x_{j+1}]$.

\section{Modeling tools}
\label{sec:modeling_tools}

\subsection{Knot insertion}
\label{sec:knot_ins}
In this section we discuss how knot insertion can simply be performed working with the transition functions.
In particular we will see that, when a MD-spline space is obtained from another by insertion of one knot,
the coefficients relating the B-spline bases of the two spaces can straightforwardly be determined by means of transition functions.
From the properties of the transition functions, it also follows that the knot-insertion coefficients are positive (see Proposition \ref{prop:ki_f}), which has a number of important consequences, namely total positivity of the B-spline basis, variation diminution and the existence of corner cutting algorithms for the constructed MD-splines. These properties can be proved replicating the same outline of their conventional spline counterpart \cite{dBDV1985,LR1983}.

Note that inserting a knot in $\bDelta_s^*$ entails that a knot is also inserted in $\bDelta_t^*$ at the same location and viceversa.
We will thus adopt the convention that \qtext{knot insertion} is intended as insertion of a knot in  $\bDelta_s^*$, bearing in mind that we could analogously reason in terms of $\bDelta_t^*$.
In particular, let $S(\PPP_{\mathbf{d}},\KKK,\bDelta)$ be a MD-spline space with associated left extended partition $\bDelta^*_s=\{s_i\}_{i=1}^K$ and let us insert one knot $\hat{s}$ in $\bDelta^*_s$, $s_{\ell} \leq \hat{s} < \min(s_{\ell+1},b)$.
Knot insertion yields a new left extended partition $\hat{\bDelta}^*_s=\{\hat{s}_i\}_{i=1}^{K+1}$ and a new spline space
$S(\PPP_{\hat{\mathbf{d}}},\hat{\KKK},\hat{\bDelta})$ such that
$S(\PPP_{\mathbf{d}},\KKK,\bDelta) \subset S(\PPP_{\hat{\mathbf{d}}},\hat{\KKK},\hat{\bDelta})$.
If $x_j<\hat{s}<x_{j+1}$ (being $x_j$ the break-points of $S(\PPP_{\mathbf{d}},\KKK,\bDelta)$), then we shall assume that the interval $[x_j,x_{j+1}]$ is divided in two subintervals of degree $d_j$ and thus, in the new space, a spline will be continuous of order $d_j-1$ at $\hat{s}$.
The following proposition provides an explicit expression for the knot insertion coefficients.

\begin{prop}\label{prop:ki_f}
Let $S(\PPP_{\mathbf{d}},\KKK,\bDelta)$  and $S(\PPP_{\hat{\mathbf{d}}},\hat{\KKK},\hat{\bDelta})$ be MD-spline spaces with left extended partitions $\bDelta^*_s=\{s_i\}_{i=1}^K$ and $\hat{\bDelta}^*_s=\{\hat{s}_i\}_{i=1}^{K+1}$ respectively. Let $S(\PPP_{\hat{\mathbf{d}}},\hat{\KKK},\hat{\bDelta})$ be obtained from $S(\PPP_{\mathbf{d}},\KKK,\bDelta)$ by insertion of a knot
$\hat{s}$ in $\bDelta^*_s$, $s_{\ell} \leq \hat{s} < \min(s_{\ell+1},b)$.
Denoting by $f_{i}$ and $\hat{f}_{i}$ the transition functions of the two spline spaces, there exist coefficients 
$0\leq \alpha_{i} \leq 1$, $i=1,\dots,K$, such that
\begin{equation}\label{eq:ki_f}
f_{i} = \alpha_{i}\,\hat{f}_{i} + (1-\alpha_{i})\,\hat{f}_{i+1}, \quad i=1,\dots,K.
\end{equation}
In particular,
\begin{equation}\label{eq:ki_coef}
\alpha_{i}=\left\{
\begin{array}{ll}
1, & i \leq \ell-d_j, \quad \mathrm{with} \quad x_j \leq \hat{s} < x_{j+1},\\
{\displaystyle\frac{D_+^{k_{i}^s+1}f_{i}({s_{i}})}{D_+^{k_{i}^s+1}\hat{f}_{i}({s_{i}})}}, & \ell-d_j+1 \leq i \leq \ell-r+1, \\
0, & i \geq \ell-r+2, \\
\end{array} \right.
\end{equation}
where $k_{i}^s$ is defined in \eqref{eq:KisKit} and $r$ is the multiplicity of $\hat{s}$ in $\hat{\bDelta}^*_s$.
\end{prop}

\begin{proof}
Recalling that $f_{i}$ is identically zero to the left of $s_i$ and identically one to the right of $t_{i-1}$,
it is immediately seen that $f_{i}$ can be a combination of $\hat{f}_{i}$ and $\hat{f}_{i+1}$ only, namely
\begin{equation}\label{eq:ki_gen}
f_{i}=\alpha_{i} \hat{f}_{i}+\beta_{i}\hat{f}_{i+1}, \quad i=1,\ldots,K.
\end{equation}

In view of \eqref{eq:Taylor}, the coefficients $\alpha_{i}$ are obtained by differentiating \eqref{eq:ki_gen} $k_{i}^s+1$ times
and evaluating at $s_i$, which yields \eqref{eq:ki_coef}.
Similarly, differentiating \eqref{eq:ki_gen} $k_{i-1}^t+1$ times
and evaluating at $t_{i-1}$ we get
\begin{equation}\label{eq:ki_coef2}
\beta_{i}=\left\{
\begin{array}{ll}
0, & i \leq \ell-d_j, \quad \textrm{with} \quad x_j \leq \hat{s} < x_{j+1}, \\
{\displaystyle\frac{D_-^{k_{i-1}^t+1}f_{i}({t_{i-1}})}{D_-^{k_{i-1}^t+1}\hat{f}_{i+1}({t_{i-1}})}}, & \ell-d_j+1 \leq i \leq \ell-r+1, \\
1, & i \geq \ell-r+2,
\end{array} \right.
\end{equation}
where $r$ is the multiplicity of $\hat{s}$ in $\hat{\bDelta}^*_s$.
Relations \eqref{eq:ki_coef} and \eqref{eq:ki_coef2} together with \eqref{eq:Taylor} show that $\alpha_{i},\beta_{i}\geq 0$.
In addition, for $\bar{x}=t_{i-1}$ we have
\[
f_i(\bar{x})=\hat{f}_i(\bar{x})=\hat{f}_{i+1}(\bar{x})=1,
\]
and thus, from \eqref{eq:ki_gen}, $\alpha_{i}+\beta_{i}=1$. The last two observations entail that $0\leq\alpha_{i}\leq1$.
\end{proof}

\begin{cor}\label{prop:ki_N}
Under the assumptions of Proposition \ref{prop:ki_f}, the B-spline bases $\{N_{i,m}\}_{i=1}^K$ and $\{\hat{N}_{i,m}\}_{i=1}^{K+1}$
are related through
\begin{equation}\label{eq:kiN}
N_{i,m} = \alpha_{i} \hat{N}_{i,m} + (1-\alpha_{i+1}) \hat{N}_{i+1,m}, \quad i=1,\dots,K,
\end{equation}
with coefficients $\alpha_{i}$ given by \eqref{eq:ki_coef}.
\end{cor}

\begin{proof}
The statement follows from relation \eqref{eq:N} by writing
\begin{align}
N_{i,m} = f_{i}-f_{i+1} & =\alpha_{i} \hat{f}_{i} + (1-\alpha_{i}) \hat{f}_{i+1} -\alpha_{i+1}\hat{f}_{i+1} -(1-\alpha_{i+1}) \hat{f}_{i+2} \\
& =\alpha_{i}(\hat{f}_{i}-\hat{f}_{i+1}) + \hat{f}_{i+1} - \hat{f}_{i+2} -\alpha_{i+1}(\hat{f}_{i+1}-\hat{f}_{i+2}) \\
& =\alpha_{i} \hat{N}_{i,m} + (1-\alpha_{i+1}) \hat{N}_{i+1,m}.
\end{align}
\end{proof}

\begin{cor}
\label{prop:ki_spl}
Under the assumptions of Proposition \ref{prop:ki_f}, a MD-spline can be expressed as
$$
f(x) = \sum_{i=1}^{K} c_i\,N_{i,m}(x) = \sum_{i=1}^{K+1} \hat{c}_i\,\hat{N}_{i,m}(x), \quad x\in[a,b].
$$
where
\begin{equation}\label{eq:de_spl_coeff}
\hat{c}_i =
\begin{cases}
c_i, & i \leq \ell-d_j, \quad \mathrm{with} \quad x_j \leq \hat{s} < x_{j+1},\\
\alpha_{i}\,c_i + (1-\alpha_{i})\,c_{i-1}, &
\ell-d_j+1 \leq i \leq \ell-r+1,\\
c_{i-1}, & i \geq \ell-r+2,
\end{cases}
\end{equation}
with $\alpha_{i}$, $i=\ell-d_j+1,\dots,\ell-r+1$, as in \eqref{eq:ki_coef}.
\end{cor}

\smallskip

The proof of Corollary \ref{prop:ki_spl} follows immediately from Corollary \ref{prop:ki_N} and equation \eqref{eq:ki_gen}.

\subsection{Degree elevation}
\label{sec:order_elev}
MD-splines allow for performing degree elevation locally, namely we can raise the degree of one (or more) section space(s) only, maintaining the other degrees unchanged. This local degree elevation represents a major and important difference with respect to conventional splines, where elevating the degree is a global operation.
Moreover, as we will illustrate by way of an example in Section \ref{sec:examples}, this feature allows for modeling parametric curves with the least number of control points.

Let $S(\PPP_{\mathbf{d}},\KKK,\bDelta)$ be a MD-spline space  having degree sequence $\mathbf{d}=(d_0,\dots,d_q)$ and dimension $K$.
It is sufficient to discuss the case where we want to elevate the degree on a single interval, say $[x_j,x_{j+1}]$, from $d_j$ to $d_{j}+1$.
Let  $S(\PPP_{\hat{\mathbf{d}}},\KKK,\bDelta)$ be the degree-elevated space, which will have dimension $K+1$.
As we will see, each B-spline basis function in $S(\PPP_{\mathbf{d}},\KKK,\bDelta)$ which is nonzero in $[x_j,x_{j+1}]$ will be expressed as a combination of two new ones. At the same time, all B-spline basis functions in $S(\PPP_{\mathbf{d}},\KKK,\bDelta)$ which are zero on $[x_j,x_{j+1}]$ carry over to the degree elevated space as well.
Thus degree elevation only requires to recompute a limited number of basis functions.

\begin{prop}\label{prop:de_basis}
Let $\bDelta_s^*$ and $\bDelta_t^*$ be the left and right extended knot partitions for $S(\PPP_{\mathbf{d}},\KKK,\bDelta)$.
By degree elevating the MD-spline space $S(\PPP_{\mathbf{d}},\KKK,\bDelta)$  in the interval $[x_j,x_{j+1}]$, that is $\hat {d}_j=d_j+1$,
we obtain the new space
$S(\PPP_{\hat{\mathbf{d}}},\KKK,\bDelta)$ and the new knot partitions $\hat{\bDelta}_s^*$ and $\hat{\bDelta}_t^*$.
Denoting by $f_{i}$, $i=1,\dots,K$, and $\hat{f}_{i}$, $i=1,\dots,K+1$, the respective transition functions, there exist
coefficients $\alpha_{i}$, with $0\leq \alpha_{i} \leq 1$, $i=1,\dots,K$, such that
\begin{equation}\label{eq:de_basis}
f_{i} = \alpha_{i}\,\hat{f}_{i} + (1-\alpha_{i})\,\hat{f}_{i+1}, \quad i=1,\dots,K.
\end{equation}
In particular,
\begin{equation}\label{eq:de_coef}
\alpha_{i}=\left\{
\begin{array}{ll}
1, & i \leq \ell-d_j, \\
{\displaystyle\frac{D_+^{k_{i}^s+1}f_{i}({s_{i}})}{D_+^{k_{i}^s+1}\hat{f}_{i}({s_{i}})}}, & \ell-d_j+1 \leq i \leq \ell, \\
0, & i \geq \ell+1,
\end{array} \right.
\end{equation}
where $k_{i}^s$ is defined in \eqref{eq:KisKit} and $\ell$ is such that $s_\ell \leq x_j< \min(s_{\ell+1},b)$.
\end{prop}
\begin{proof}
We start by observing that
\begin{equation}\label{eq:de_gen}
f_{i}=\alpha_{i} \hat{f}_{i}+\beta_{i}\hat{f}_{i+1}, \quad i=1,\ldots,K,
\end{equation}
as, if $f_{i}$ was a combination of other basis functions, it could not have at
$s_i$ and $t_{i-1}$ the right continuities.
 In view of \eqref{eq:Taylor}, the coefficients $\alpha_{i}$ can be obtained by differentiating $k_{i}^s+1$ times expression \eqref{eq:de_gen} and evaluating the result at $s_i$. Analogously, the coefficients $\beta_{i}$ can be obtained by differentiating $k_{i-1}^t+1$ times expression \eqref{eq:de_gen} and evaluating the result at $t_{i-1}$.
 Hence the statement follows from the same arguments in the proof of Poposition \ref{prop:ki_f}.
\end{proof}

\begin{cor}\label{prop:de_N}
Under the assumptions of Proposition \ref{prop:de_basis}, the B-spline bases $\{N_{i,m}\}_{i=1}^K$ of $S(\PPP_{\mathbf{d}},\KKK,\bDelta)$ and $\{\hat{N}_{i,\hat{m}}\}_{i=1}^{K+1}$ of $S(\PPP_{\hat{\mathbf{d}}},\KKK,\bDelta)$
are related through
\begin{equation}\label{eq:kiN}
N_{i,m} = \alpha_{i} \hat{N}_{i,\hat{m}} + (1-\alpha_{i+1}) \hat{N}_{i+1,\hat{m}},
\end{equation}
with coefficients $\alpha_{i}$ given in \eqref{eq:de_coef}.
\end{cor}

\begin{proof}
The proof follows the same outline of Corollary \ref{prop:ki_N}.
\end{proof}

\begin{cor}\label{prop:de_spl}
Under the assumptions of Proposition \ref{prop:de_basis}, a MD-spline can be expressed as
$$
f(x) = \sum_{i=1}^{K} c_i\,N_{i,m}(x) = \sum_{i=1}^{K+1} \hat{c}_i\,\hat{N}_{i,\hat m}(x), \quad  x\in[a,b].
$$
where 
\begin{equation}\label{eq:de_spl_coeff}
\hat{c}_i =
\begin{cases}
c_i, & i \leq \ell-d_j,\\
\alpha_{i}\,c_i + (1-\alpha_{i})\,c_{i-1}, &
\ell-d_j+1 \leq i \leq \ell,\\
c_{i-1}, & i \geq \ell+1,
\end{cases}
\end{equation}
with $\alpha_{i}$, $i=\ell-d_j+1,\dots,\ell$, as in \eqref{eq:de_coef}.
\end{cor}

\smallskip

The proof of Corollary \ref{prop:de_spl} follows immediately from Corollary \ref{prop:de_N} and equation \eqref{eq:de_basis}.

\section{Geometrically continuous MD-splines}
\label{sec:geometric}
In this section we explore a possible extension of the proposed MD framework to the context of
geometrically continuous splines.
These splines allow for relaxing the strict requirement for parametric continuity and introduce more \qtext{degrees of freedom}, which can be used to intuitively modify the shape of parametric curves \cite{DynMicchelli1989}.
To the best of our knowledge, all instances of geometrically continuous splines appeared so far are featured by having section spaces all of the same degree or dimension.
Our main objective is to demonstrate how our approach based on transition functions is easily extendible and adaptable to construct geometrically continuous MD-splines.
To this end, we will content ourselves to introducing the basic idea postponing a thorough study on the subject to a future work. A numerical example will be provided in the next section.

In addition to the setting and notation adopted so far,
we now also associate with the elements of $\bDelta$
a sequence $\bM \coloneqq (M_1,\ldots,M_q)$ of \emph{connection matrices}, where $M_i$, $i=1,\dots,q$, is lower
triangular of order $k_i+1$, has positive diagonal entries and first row and column equal to $(1,0,\dots,0)$.
Hence \emph{geometrically continuous MD-splines} are functions which satisfy a modified version of Definition \ref{def:QCS},
where condition ii) is replaced by the following one:
\begin{itemize}
\item[ii)]
$\displaystyle{
M_i \left(D^0 p_{i-1}(x_i),D^1 p_{i-1}(x_i),\dots,D^{k_i} p_{i-1}(x_i)\right)^T =
\left(D^0 p_i(x_i),D^1 p_i(x_i),\dots,D^{k_i} p_i(x_i)\right)^T, \quad i=1,\dots,q.
}$
\end{itemize}
We will denote the spline space by $S(\PPP_{\mathbf{d}},\KKK,\bDelta,\bM)$.
It shall be observed that, when all matrices $M_i$ are the identity, the considered MD-splines are parametrically continuous
(and thus they fall into the framework of the previous sections).
In addition, the continuity of these splines is guaranteed by the requirement that the first row and column of each matrix $M_i$ be equal to the vector $(1,0,\dots,0)$.
Without delving into details, we recall that the introduction of the connection matrices does not alter the dimension of the spline space.

Transition functions for geometrically continuous spline spaces can be constructed by generalizing the procedure in Section \ref{sec:trans_fun}. More precisely, it is sufficient to require that each transition function $f_i$, $i=2,\dots,K$, satisfy a modified version of the parametric continuity conditions \eqref{eq:cond_tf}, obtained by replacing the second row in \eqref{eq:cond_tf} with:
\[M_j \left(D^0_- f_{i}(x_{j}),D^1_- f_{i}(x_{j}),\dots,D^{k_{j}}_- f_{i}(x_{j})\right)^T =
\left(D^0_+ f_{i}(x_{j}),D^1_+ f_{i}(x_{j}),\dots,D^{k_{j}}_+ f_{i}(x_{j})\right)^T,
\quad j=ps_i+1,\dots,pt_{i-1}-1.\]
Similarly as discussed in Section \ref{sec:trans_fun}, such a modified version of  \eqref{eq:cond_tf} yields a linear system which can be solved for determining the coefficients of the transition function $f_i$, for all $i=2,\dots,K$.
Moreover, knot insertion and degree elevation can be performed exploiting the transition functions, analogously as described in Section \ref{sec:modeling_tools}.

\section{Numerical examples}
\label{sec:examples}
In this section we present two numerical examples. The first demonstrate the use of the tools introduced in the previous sections,
thus highlighting the advantages and potentials of MD-splines for geometric modeling.
The second example is conceived to illustrate geometrically continuous MD-splines, which, to the best of our knowledge, have never considered in the previous proposals. These splines combine the benefits of the multi-degree framework and of geometric continuity.

\paragraph{Geometric modeling with MD-splines}
To illustrate the potential of MD-splines in geometric modeling, we consider a typical modeling section, where knot insertion and degree elevation are used in order to add details to a basic initial shape (Fig.\ \ref{fig:turtle}).
By means of MD-splines, the target curve can be represented relying on the lowest possible degree on each interval and thus using the minimum number of control points.

\begin{itemize}[leftmargin=*]
\item Fig.\ \ref{fig:TA} depicts a parametric curve from the MD-spline space in Example \ref{ex:running}. The control points are marked by circles (the first and last point are coincident and marked with a double circle), whereas the black dots identify the junction of two spline segments. We recall from Example \ref{ex:running} that $[a,b]=[0,7]$, $\bDelta =\{x_1,x_2,x_3\}=\{ 1, 3, 6\}$, $\mathbf{d} = (d_0,\dots,d_3)=(1,2,4,2)$ and $\mathcal{K} = (k_1,k_2,k_3)=(0,1,2)$, while we refer to Fig.\ \ref{fig:ex1_part} for the two extended partitions $\bDelta^*_s$ and $\bDelta^*_t$.

\item \emph{Knot insertion} -- In Fig.\ \ref{fig:TB} a knot is inserted in the partitions $\bDelta^*_s$ and $\bDelta^*_t$ at $2.6$, thus obtaining $\hat{\bDelta}^*_s=\{0,0,1,1,2.6,3, 3, 3\}$ and $\hat{\bDelta}^*_t=\{1,2.6,3,6,6,7,7,7\}$.
    Accordingly, the break-point sequence becomes $\hat{\bDelta}=\{1,2.6,3,6\}$. The interval $[x_1,x_2]$ in $\bDelta$ is split in two intervals of degree $2$ in $\hat{\bDelta}$, obtaining a new degree vector $\hat{\mathbf{d}}=(1,2,2,4,2)$. Finally, because a single knot is inserted between two degree-$2$ intervals, the continuity vector needs to be updated to $\hat{\mathcal{K}}=(0,1,1,2)$.
As a result of the knot-insertion procedure, the corresponding control polygon contains one additional point.
\item \emph{Iterated degree elevation} --
In the interval $[2.6,3]$ the degree elevation is performed three times (Fig.\ \ref{fig:TC}).
This does not change the break-point sequence $\hat{\bDelta}$  and the continuity $\hat{\mathcal{K}}$.
Since the spline degree is locally raised from 2 to 5 in the third interval, the degree vector is updated to  $\hat{\hat{\mathbf{d}}}=(1,2,5,4,2)$.
Moreover, because the continuity should not change, a suitable number of coincident knots must be added in the
left and right extended partitions, which yields $\hat{\hat{\bDelta}}^*_s=\{0,0,1,1,2.6,2.6,2.6,2.6,3, 3, 3\}$ and $\hat{\hat{\bDelta}}^*_t=\{1,2.6,3,3,3,3,6,6,7,7,7\}$. Overall three more points are added to the control polygon.

\item The control points gained through knot insertion and iterated degree elevation are used to model the turtle head in Fig.\ \ref{fig:TD}.

\item \emph{Conversion to Bézier form} -- The modelled curve is converted to Bézier form (Fig.\ \ref{fig:TBez}) as explained in the last paragraph of Section \ref{sec:trans_fun}.

\item \emph{Conversion to conventional B-spline form} -- By means of degree elevation, the curve is converted to a conventional spline of degree 5 (Fig.\ \ref{fig:TD5}). The obtained curve features 22 control points, while only 11 control points are needed to represent the same curve in MD-spline form (Fig. \ref{fig:TD}).
\end{itemize}

\paragraph{Geometric continuity}
To illustrate the benefit of modeling with geometrically continuous MD-splines, we consider the following example based on
a spline space $S(\PPP_{\mathbf{d}},\KKK,\bDelta,\bM)$ on $[a,b]=[0,3]$, with $\bDelta=\{0.75,1.75,2.5\}$, $\mathbf{d}=(3,4,3,1)$, $\KKK=(2,2,0)$. In addition,
in correspondence of $x_1$ and $x_2$, we set the connection matrices:
\begin{equation}
\label{eq:Mi}
M_1=
\begin{pmatrix}
1   &  0  &   0\\
0   &  \alpha  &    0\\
0   &  \beta  &  \gamma
\end{pmatrix},
\qquad
M_2=
\begin{pmatrix}
1   &  0  &   0\\
0    & 1/\alpha  &   0\\
0   &  \frac{\beta}{\alpha\gamma}  &   1/\gamma
\end{pmatrix}.
\end{equation}

\noindent
The entries of $M_2$ ensure that the symmetry in the data in Fig.\ \ref{fig:geom2} can be preserved for any choice of $\alpha,\beta$ and $\gamma$.
Additionally, we choose $\gamma=\beta^2$, in such a way that the curvature is continuous at $x_1$ and $x_2$ for any $\alpha$ and $\beta$.
Note that, being the curve $C^0$ only at $x_3$, we are forced to choose $M_3=1$.
In the curves displayed in the top row of Fig. \ref{fig:geom2}, the parameter $\alpha$ has the constant value $1$, whereas $\beta$ takes values $-5,0,10$ from left to right, respectively. For increasing values of $\beta$, a progressive deformation of the shape is clearly visible.
In the bottom row of Fig.\ \ref{fig:geom2}, instead, the parameter $\beta$ is fixed to $\beta=1$, whereas $\alpha$ varies in the range $\alpha=0.25,1,4$ from left to right, respectively, and the curve progressively changes its shape accordingly.
Fig.\ \ref{fig:geom2_bases} illustrates the basis functions corresponding to the curves in the bottom row of Fig.\ \ref{fig:geom2}.
Note that in $[x_0,x_3]=[0,2]$ the basis functions are symmetric about the center of the interval
and this also holds for the other bases, which are not displayed, and justifies the symmetric behavior of all the curves in Fig.\ \ref{fig:geom2}.
In addition, all the curves in the top row of Fig.\ \ref{fig:geom2}, characterized by $\alpha=1$, have parametric $C^1$ continuity, and the one in Fig.\ \ref{fig:gc2}
is even $C^2$ at $x_1$ and $x_2$.
Focusing on the bottom row, instead, we can see that the selected values of $\alpha$ and $\beta$ yield a $C^1$ curve in Fig.\
\ref{fig:gi2}, while the other two curves have geometrically continuous first and second derivatives at $x_1$ and $x_2$.

\begin{figure}
\centering
\subfigure[]{
{\includegraphics[height=0.19\textwidth]{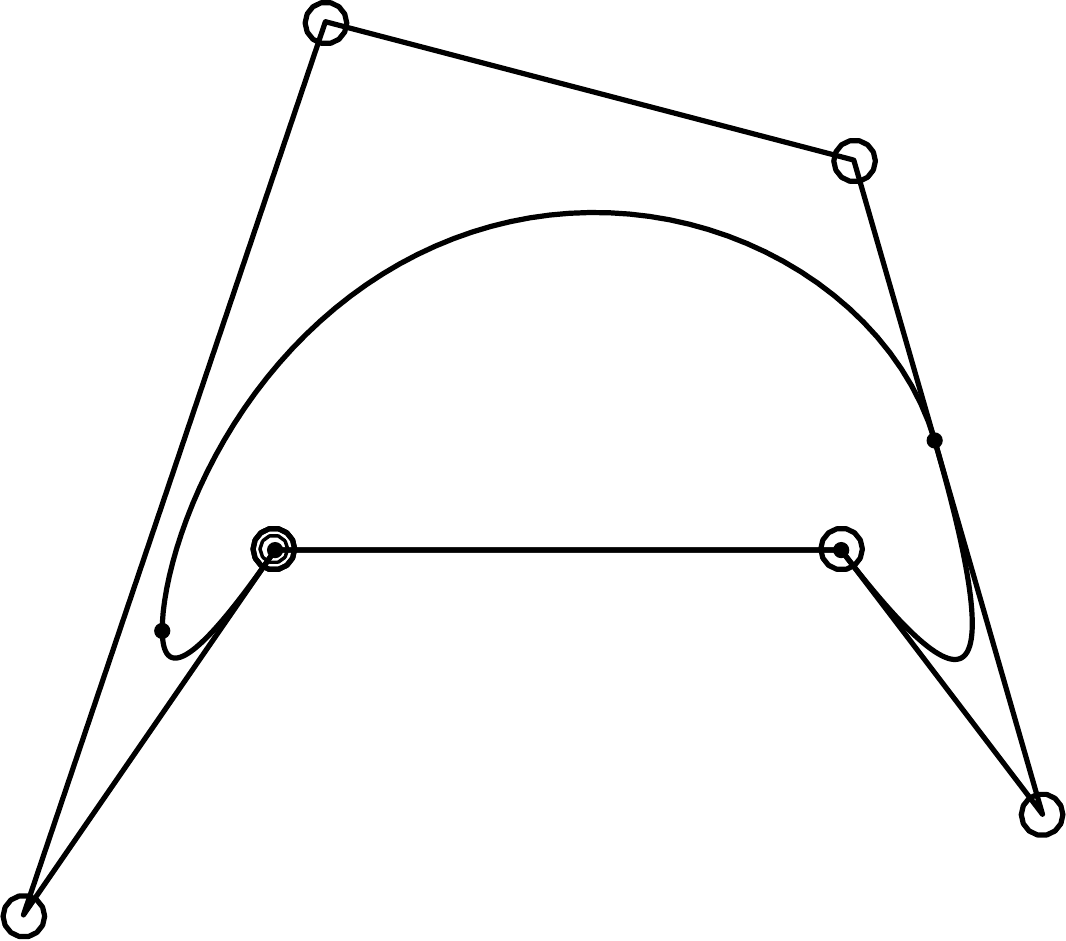}\label{fig:TA}}
}
\hfill
\subfigure[]{
{\includegraphics[height=0.19\textwidth]{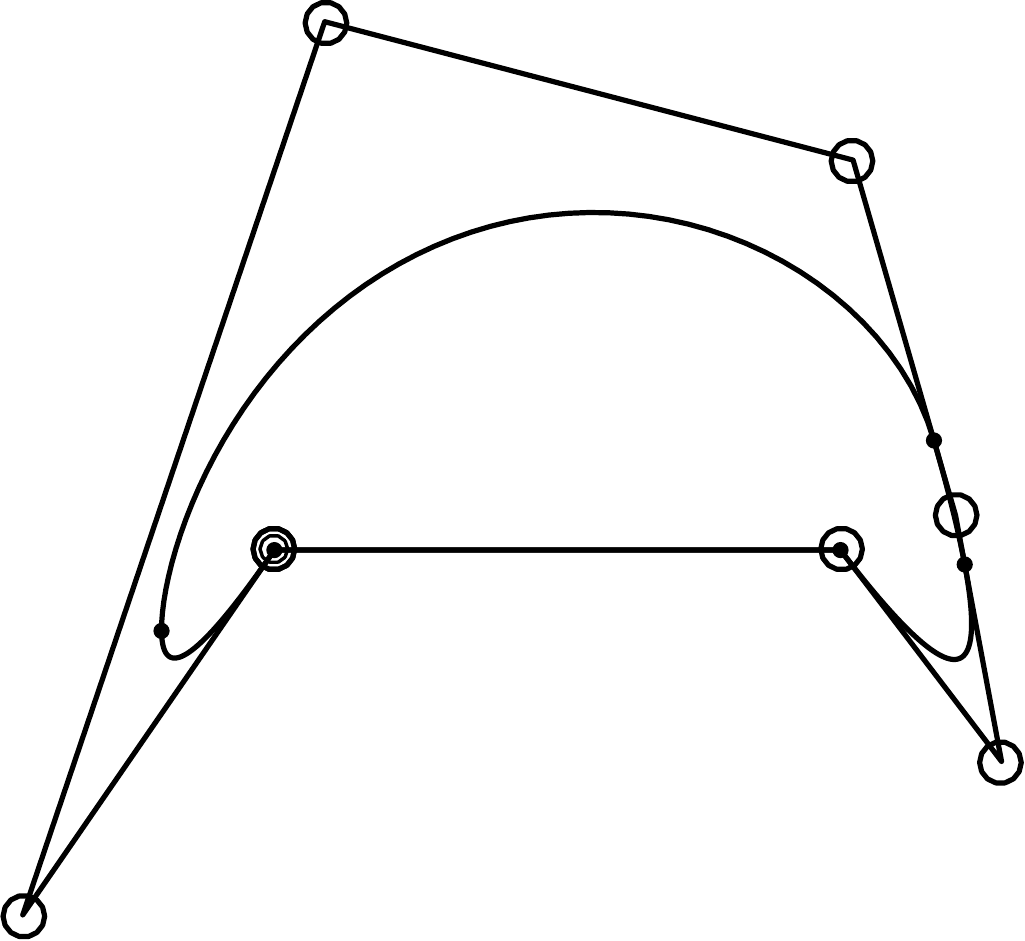}\label{fig:TB}}
}
\hfill
\subfigure[]{
{\includegraphics[height=0.19\textwidth]{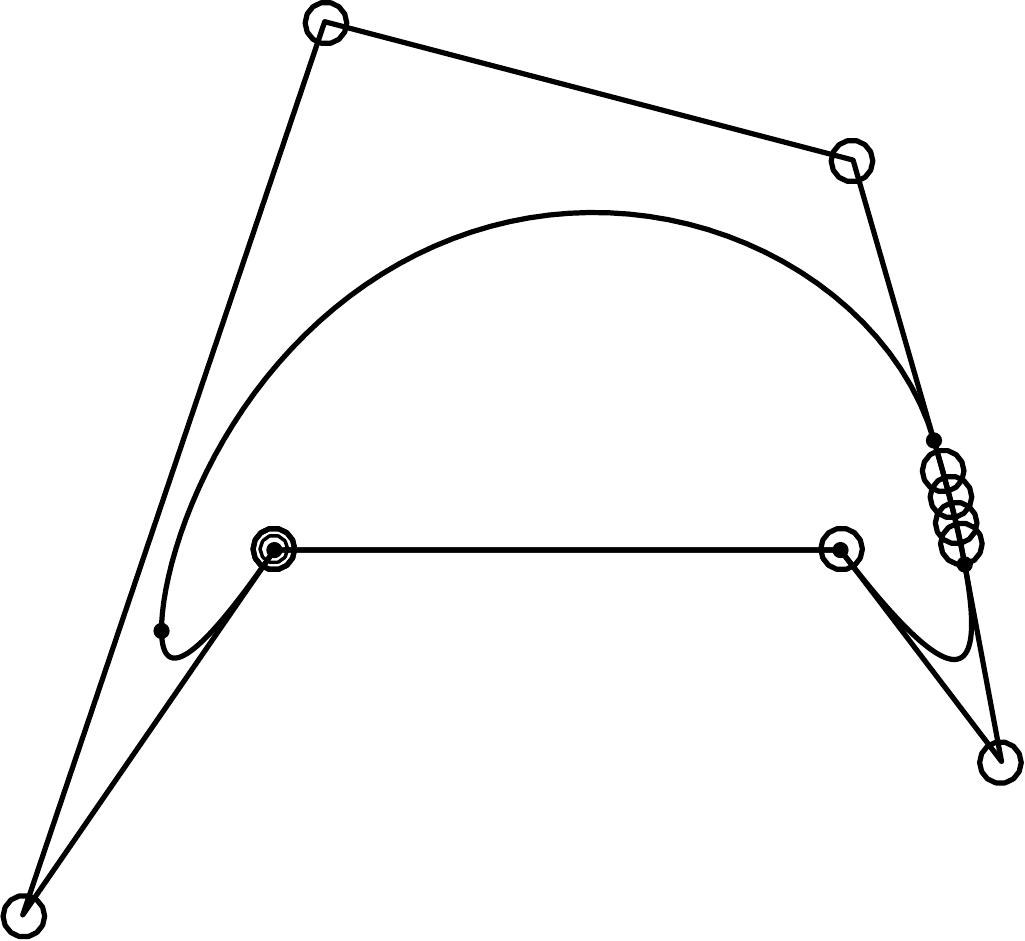}\label{fig:TC}}
}
\hfill
\subfigure[]{
{\includegraphics[height=0.19\textwidth]{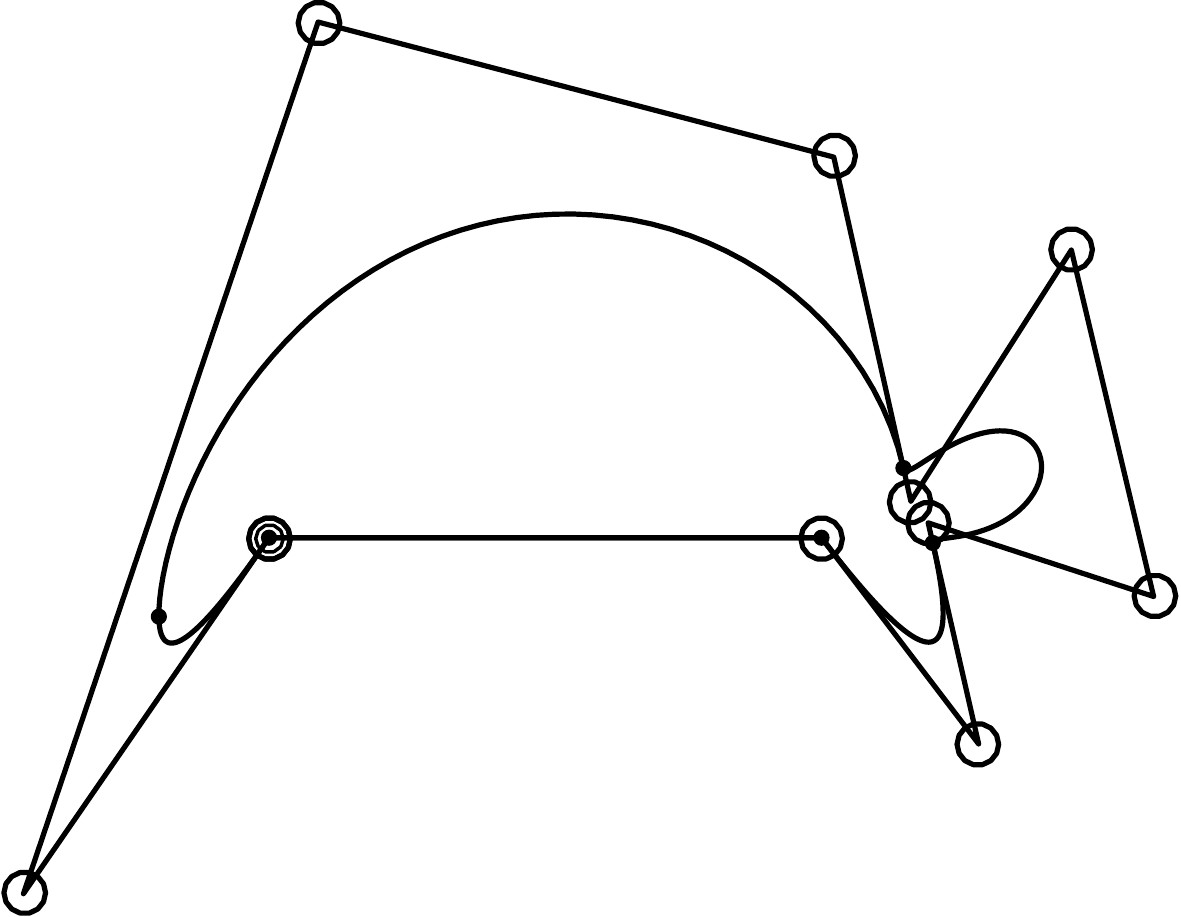}\label{fig:TD}}
}
\\
\subfigure[]{
{\raisebox{0.54cm}{\hphantom{\qquad\;\;\,}\includegraphics[height=0.165\textwidth]{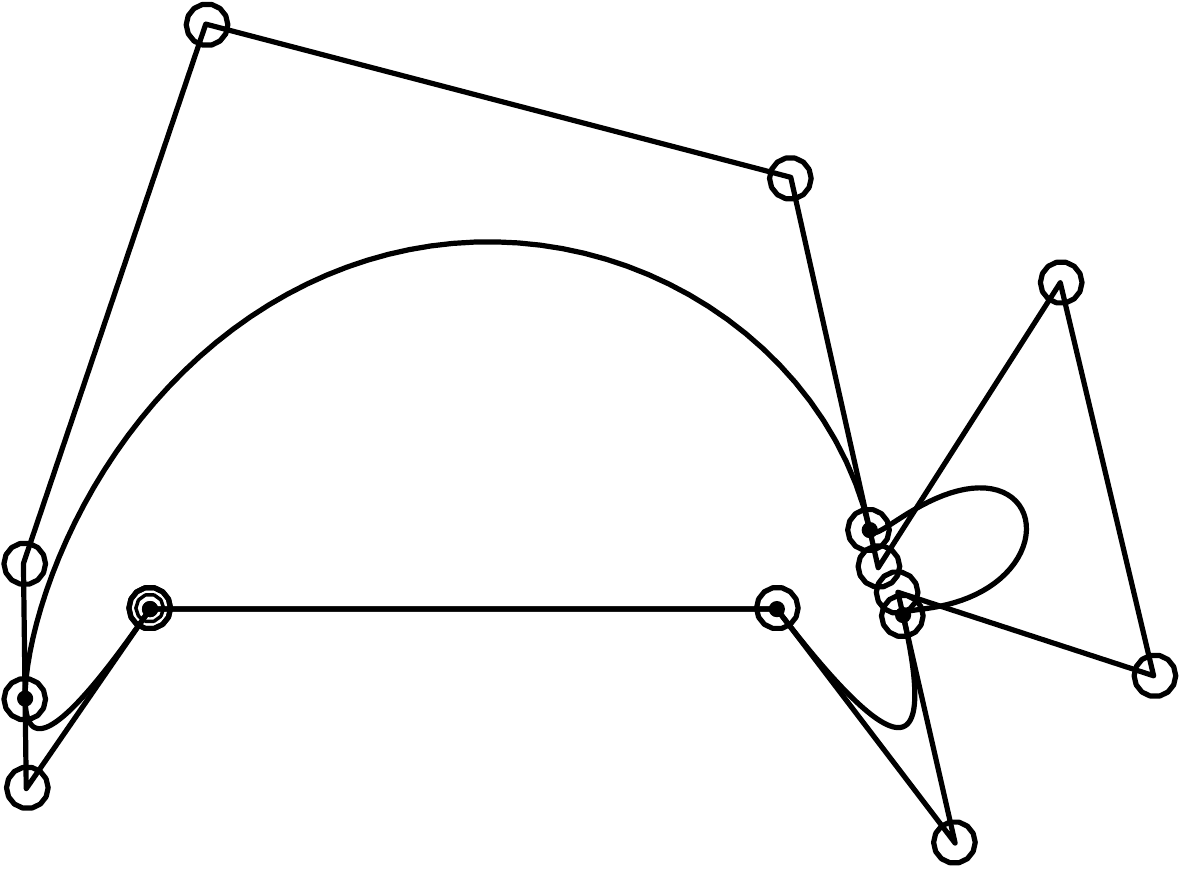}}\label{fig:TBez}}
}
\subfigure[]{
{\raisebox{0.84cm}{\hphantom{\qquad\;\;\,}\includegraphics[height=0.135\textwidth]{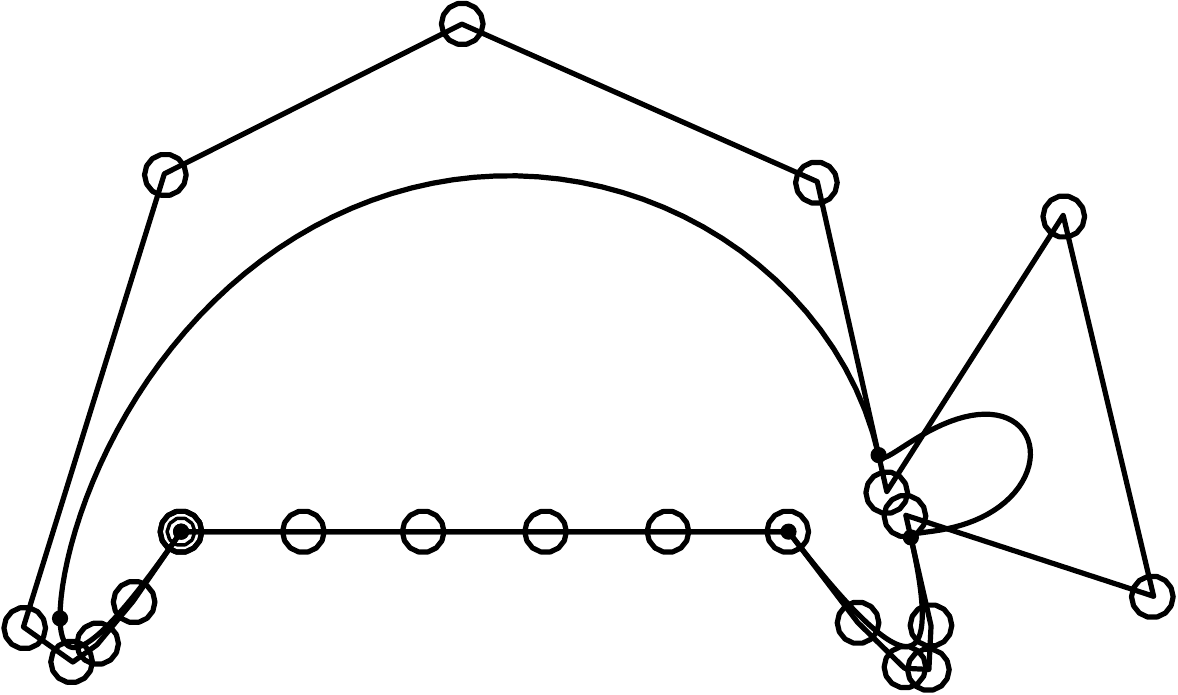}}\label{fig:TD5}}
}
\subfigure[]{
{\raisebox{0.96cm}{\hphantom{\qquad\;}\includegraphics[height=0.097\textwidth]{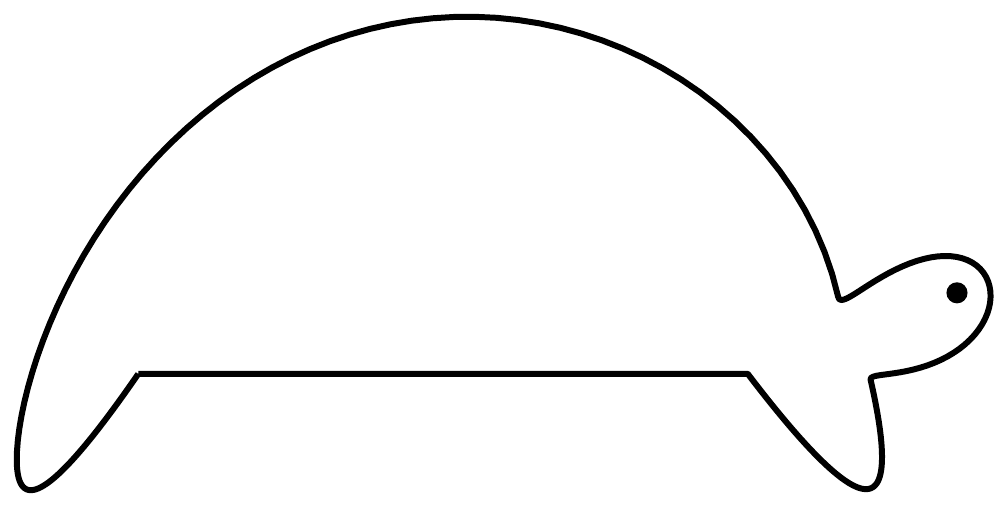}\label{fig:TDs}}}
}
\caption{Modeling with a MD-spline curve: \subref{fig:TA} parametric curve from the MD-spline space in Example \ref{ex:running}; \subref{fig:TB} curve after insertion of one knot and \subref{fig:TC} degree elevation; \subref{fig:TD} displacement of the obtained control points to model the turtle head; \subref{fig:TBez} conversion to Bézier form; \subref{fig:TD5} representation as a degree-$5$ spline.}
\label{fig:turtle}
\end{figure}

\begin{figure}
\centering
\begin{tabular}{ccc}
\subfigure[$\alpha=1$, $\beta=-5$]{
{\includegraphics[width=0.20\textwidth]{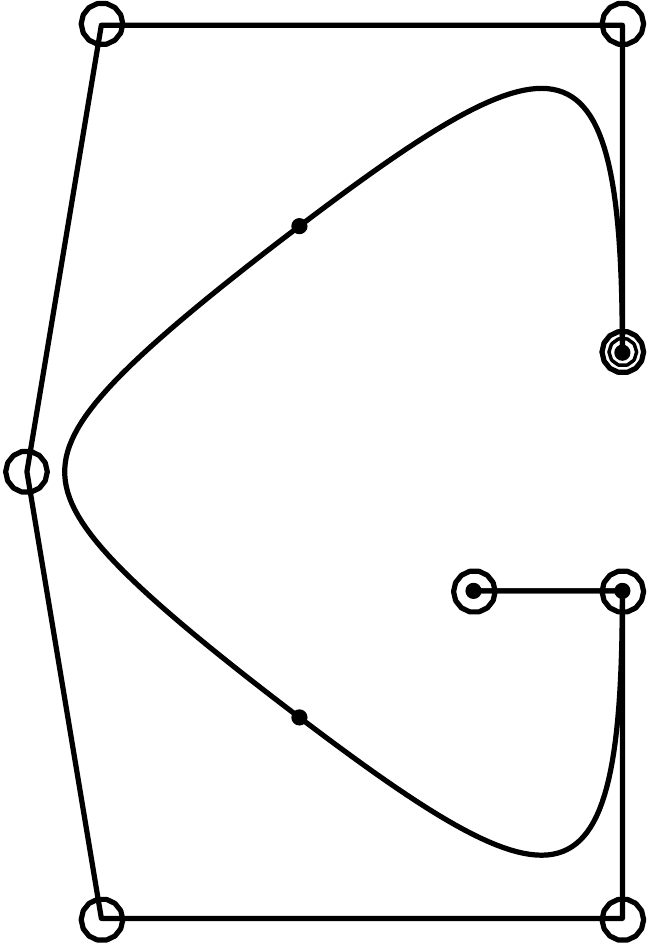}\label{fig:gb2}}
}\hspace{1.4cm}
&
\subfigure[$\alpha=1$, $\beta=0$]{
{\includegraphics[width=0.20\textwidth]{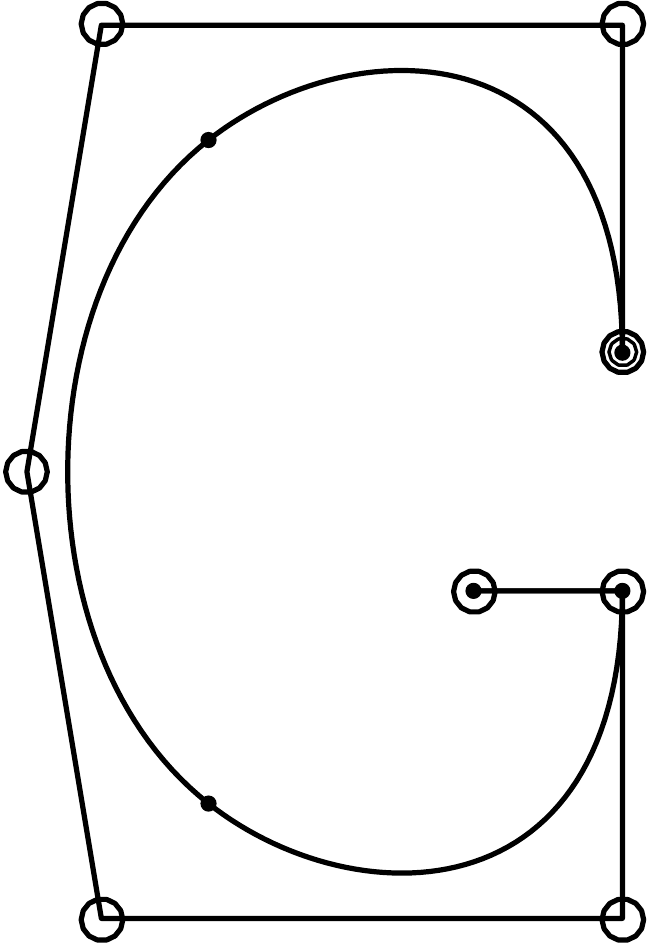}\label{fig:gc2}}
}\hspace{1.4cm}
&
\subfigure[$\alpha=1$, $\beta=10$]{
{\includegraphics[width=0.20\textwidth]{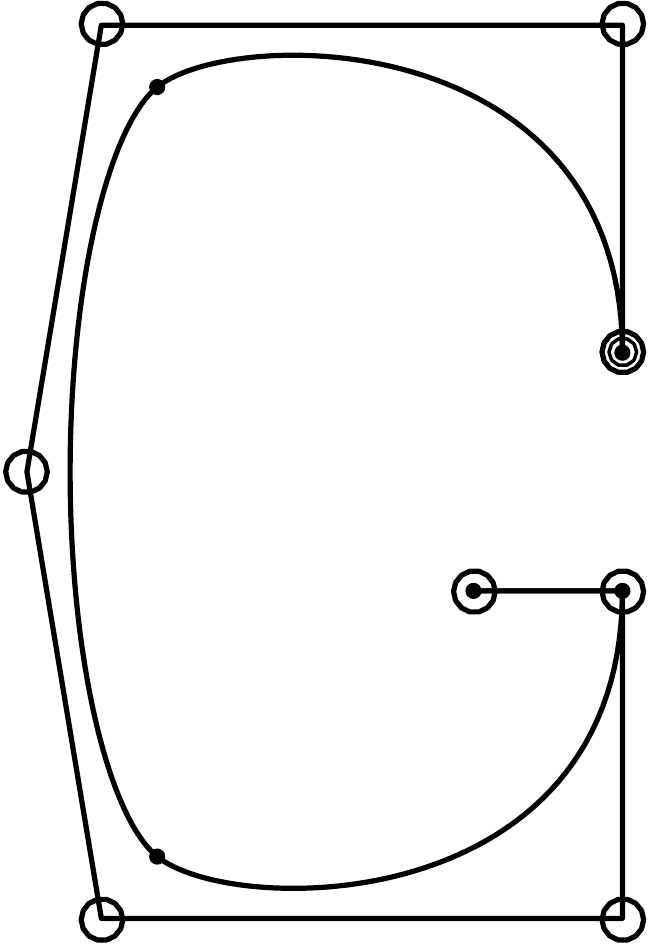}\label{fig:gd2}}
}\\
\subfigure[$\alpha=0.25$, $\beta=1$]{
{\includegraphics[width=0.2\textwidth]{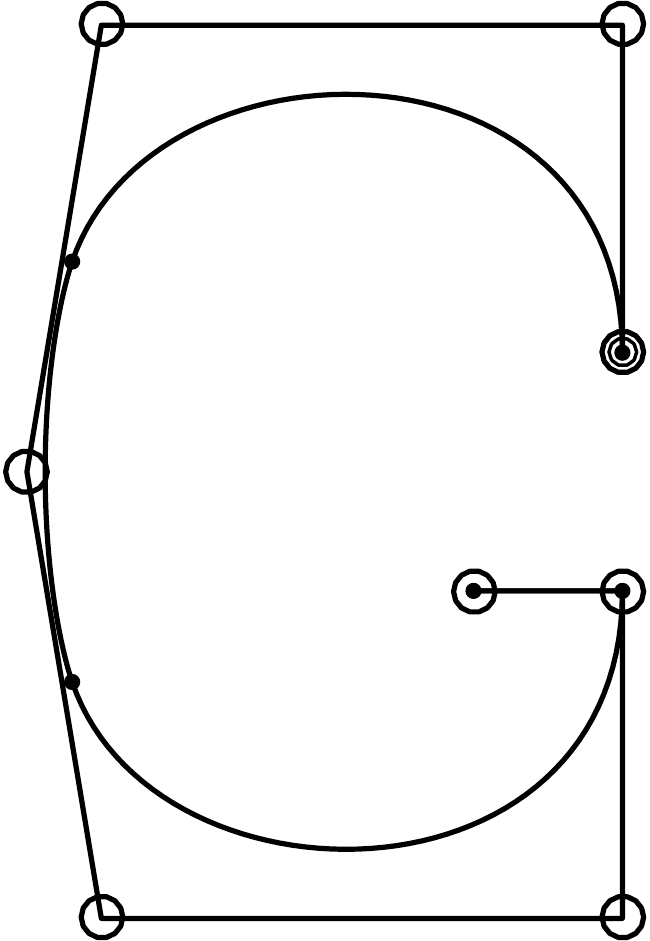}\label{fig:gh2}}
}\hspace{1.4cm}
&
\subfigure[$\alpha=1$, $\beta=1$]{
{\includegraphics[width=0.2\textwidth]{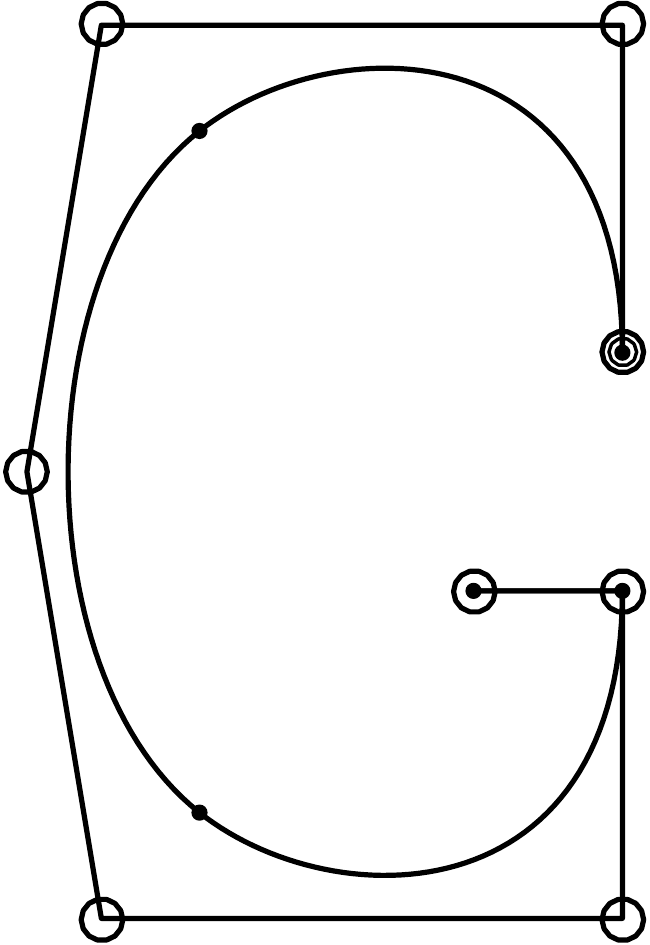}\label{fig:gi2}}
}\hspace{1.4cm}
&
\subfigure[$\alpha=4$, $\beta=1$]{
{\includegraphics[width=0.2\textwidth]{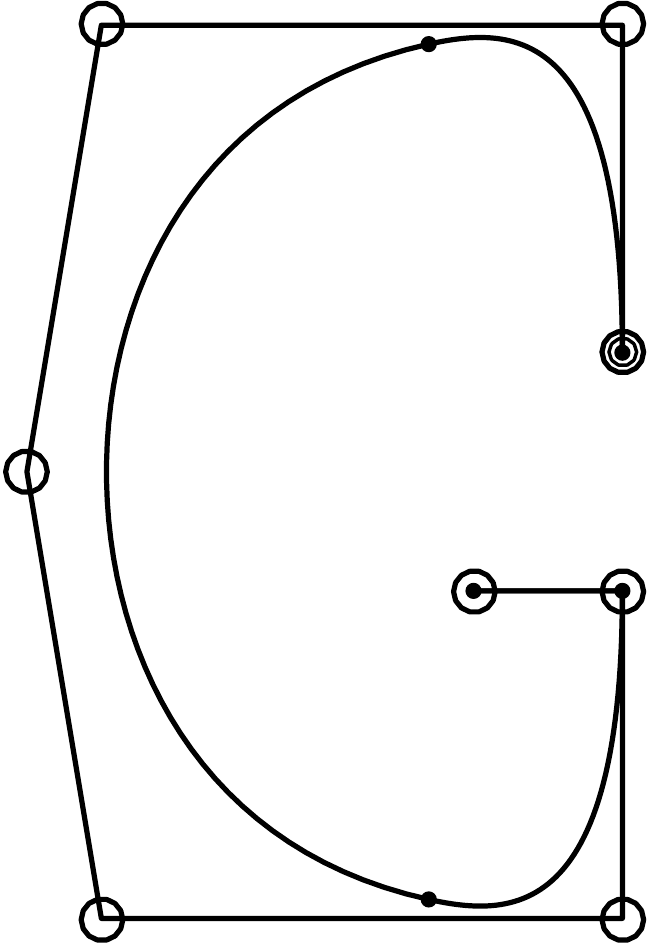}\label{fig:gl2}}
}
\end{tabular}
\caption{Example of geometric continuity.}
\label{fig:geom2}
\end{figure}

\begin{figure}
\centering
\subfigure[$\alpha=0.25$, $\beta=1$]{
{\includegraphics[width=0.3\textwidth]{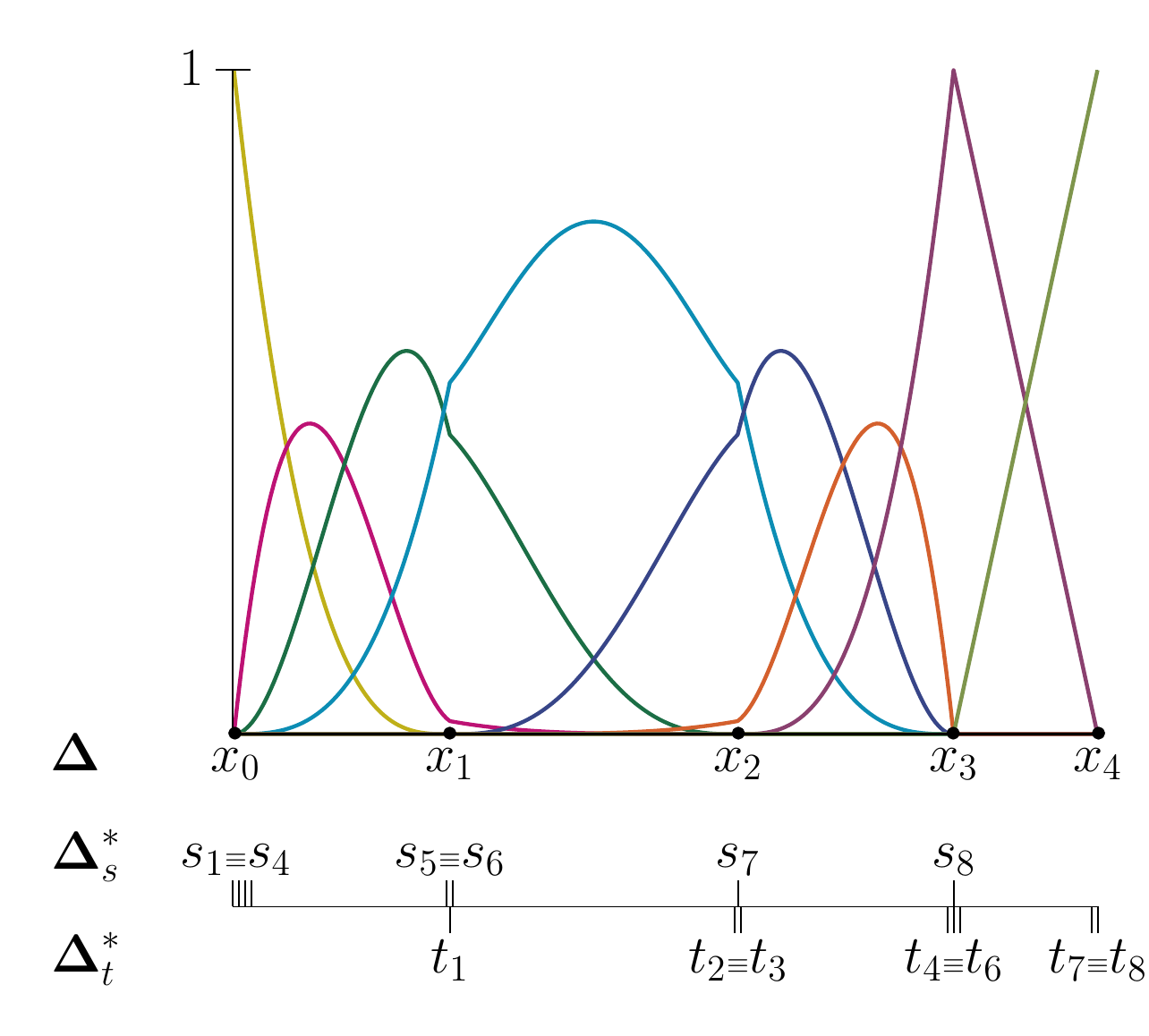}\label{fig:gh2_basis}}
}
\subfigure[$\alpha=1$, $\beta=1$]{
{\includegraphics[width=0.3\textwidth]{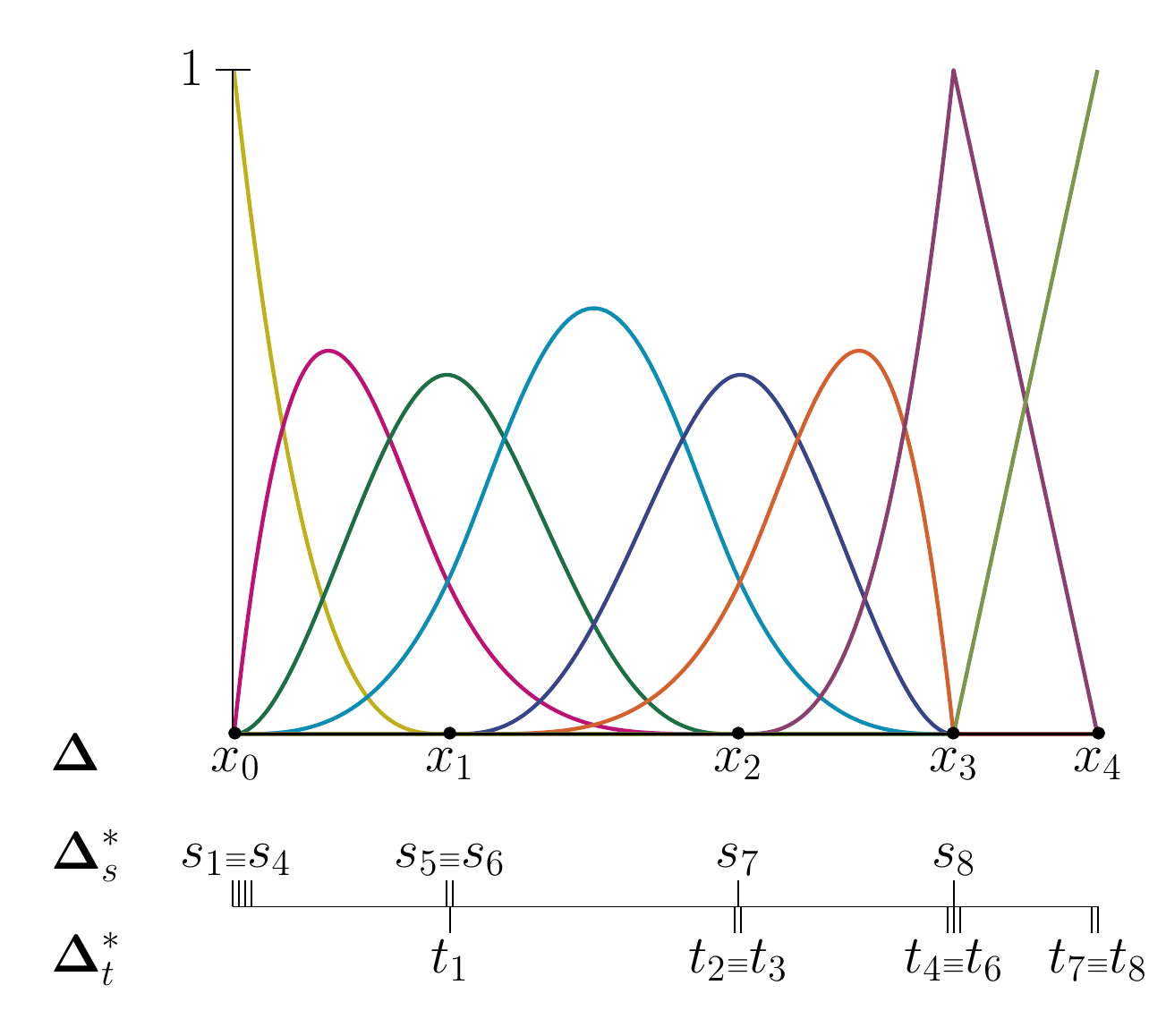}\label{fig:gi2_basis}}
}
\subfigure[$\alpha=4$, $\beta=1$]{
{\includegraphics[width=0.3\textwidth]{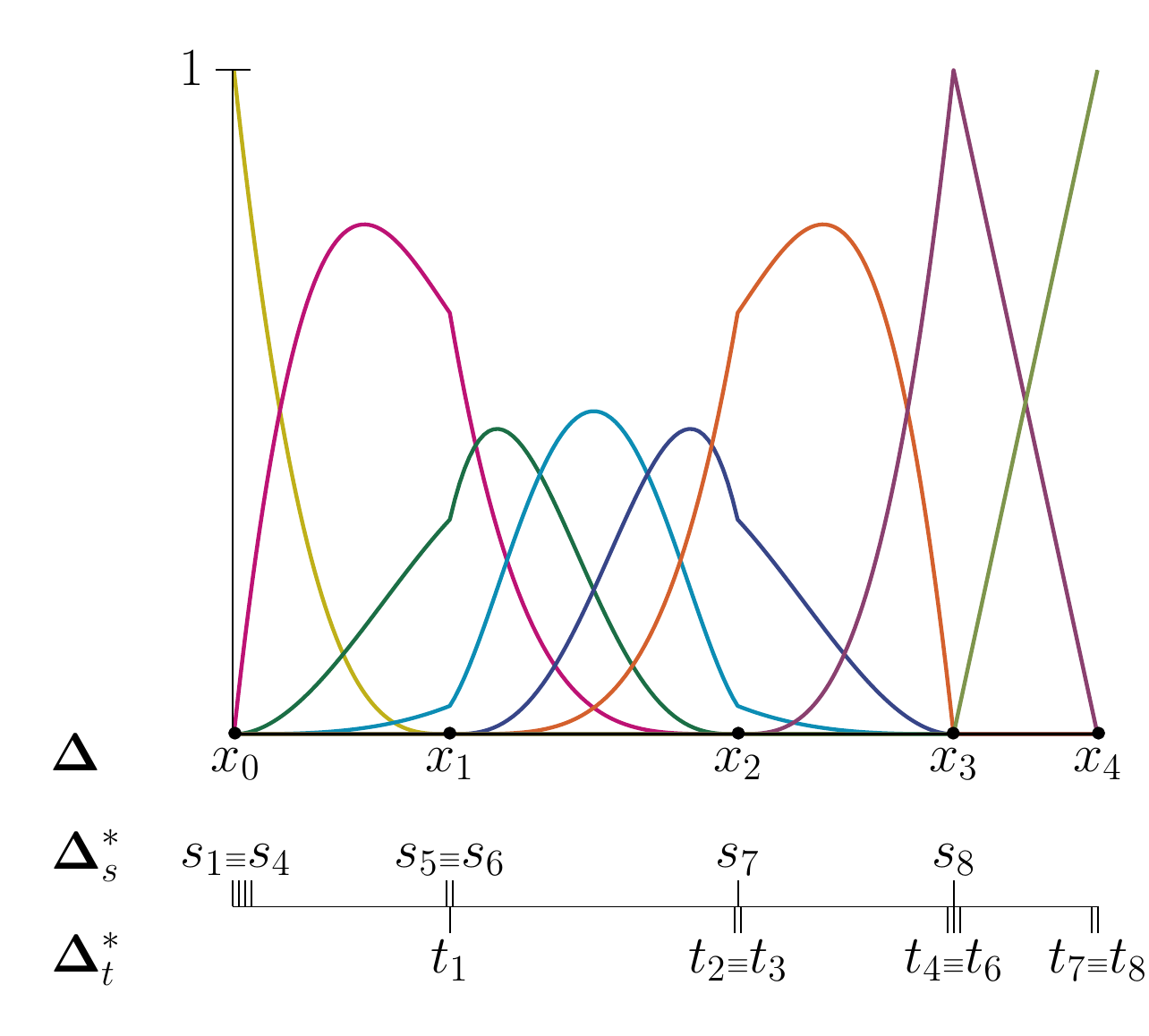}\label{fig:gl2_basis}}
}
\caption{Basis functions corresponding to the curves in the bottom row of Fig.\ \ref{fig:geom2}.}
\label{fig:geom2_bases}
\end{figure}

\section{Conclusion}
\label{sec:conclusion}
In this paper we have presented an integral relation for the construction of a multi-degree B-spline basis
which yields a wider family of splines compared to previously proposed similar approaches.
As an alternative to the integral recurrence relations, we have proposed a more convenient way to compute
the B-spline basis which relies on the use of transition functions.
Using the transition functions both knot insertion and (local) degree elevation methods can be formulated for MD-splines.
We have applied these tools to illustrate the efficiency of the proposed MD-splines in geometric modeling.
Finally, we have given some hints about
generalizing the proposed construction to the wider context of geometrically continuous splines and we plan to investigate this thoroughly in a future work.
Another interesting subject matter for future study is how to use MD-splines in surface modeling, with particular reference to T-splines and their application in isogeometric analysis.


\section*{Acknowledgements}
The authors gratefully acknowledge support from the Italian GNCS-INdAM.




\ifelsevier
\bibliographystyle{elsarticle-harv}
\bibliography{BCM_MD_splines} 
\fi




\end{document}